\documentclass[10pt, twoside, reqno]{amsart}
\usepackage{xcolor, amstext, amsfonts, amssymb, amsbsy, latexsym}
\usepackage{enumerate}
\usepackage{longtable}
\usepackage[utf8]{inputenc}
\usepackage[T1]{fontenc}
\usepackage{xy, hhline}
\usepackage{hyperref}
\usepackage{color}
\usepackage{booktabs}
\advance\hoffset by -1in
\advance\voffset by -1in

\usepackage{comment}

\newtheorem{lemma}{Lemma}[section]
\newtheorem{theo}[lemma]{Theorem}
\newtheorem{prop}[lemma]{Proposition}

\newtheorem{cor}[lemma]{Corollary}

\newtheorem{problem}[lemma]{Open Problem}
\theoremstyle{definition}
\newtheorem{notation}[lemma]{Notation}
\newtheorem{defin}[lemma]{Definition}

\newtheorem{remark}[lemma]{Remark}

\numberwithin{equation}{section}

\newenvironment{eq}{\begin{equation}}{\end{equation}}

\newcommand{\Char}{\mathop{\rm char}}

\newcommand{\bfA}{\mathbf{A}} 
\newcommand{\bfB}{\mathbf{B}} 
\newcommand{\bfC}{\mathbf{C}} 
\newcommand{\bfD}{\mathbf{D}} 
\newcommand{\bfE}{\mathbf{E}} 
\newcommand{\bfN}{\mathbf{N}} 

\newcommand{\calU}{\mathcal{U}} 
\newcommand{\calT}{\mathcal{T}}
\newcommand{\calV}{\mathcal{V}}

\newcommand{\FF}{\mathbb{F}}

\newcommand{\CC}{\mathbb{C}}

\newcommand{\NN}{\mathbb{N}}

\newcommand{\algA}{\mathcal{A}}
\newcommand{\algB}{\mathcal{B}}
\newcommand{\algV}{\mathcal{V}}
\newcommand{\algW}{\mathcal{W}}
\newcommand{\algX}[1]{{\rm alg}_{\FF}\{X\}_{#1}}

\newcommand{\Sym}{{\mathcal S}}
\newcommand{\calcP}{{\mathcal P}}

\newcommand{\Pol}{{\rm Pol}}

\newcommand{\ML}[1]{M_{\rm L}^{(#1)}} 
\renewcommand{\MR}[1]{M_{\rm R}^{(#1)}} 

\newcommand{\MdL}[2]{M_{\rm L}^{(#1,#2)}} 
\newcommand{\MdR}[2]{M_{\rm R}^{(#1,#2)}} 

\newcommand{\al}{\alpha}
\newcommand{\be}{\beta}
\newcommand{\ga}{\gamma}
\newcommand{\la}{\lambda}
\newcommand{\de}{\delta}
\newcommand{\De}{\Delta}

\newcommand{\LA}{\langle}
\newcommand{\RA}{\rangle}

\newcommand{\ov}[1]{\overline{#1}}
\newcommand{\un}[1]{{\underline{#1}} }
\newcommand{\alg}{\mathop{\rm alg}}
\newcommand{\mdeg}{\mathop{\rm mdeg}}

\newcommand{\tr}{\mathop{\rm tr}}
\newcommand{\Tr}{\mathop{\rm Tr}}

\newcommand{\Aut}{\mathop{\rm Aut}}

\newcommand{\re}{\mathop{\rm Re}}
\newcommand{\im}{\mathop{\rm Im}}

\newcommand{\G}{{\rm G}_2}

\newcommand{\GL}{{\rm GL}}
\newcommand{\PGL}{{\rm PGL}}

\newcommand{\matr}[4]{\left(\begin{array}{cc}
#1 & #2 \\
#3 & #4 \\
\end{array}\right)}

\newcommand{\smatr}[4]{\left(\begin{smallmatrix}
#1 & #2 \\
#3 & #4 \\
\end{smallmatrix}\right)}

\newcommand{\bigfrac}[2]{\!\!\!\!\begin{array}{l}
#1 \\
#2 \\
\end{array}}

\newcommand{\OO}{\mathbf{O}}
\newcommand{\uu}{\mathbf{u}}
\newcommand{\vv}{\mathbf{v}}
\newcommand{\cc}{\mathbf{c}}
\newcommand{\zero}{\mathbf{0}}

\begin{document}
\title[Polynomial invariants for low dimensional algebras]{Polynomial invariants for low dimensional algebras}

\thanks{This research was supported by MINEDUC-UA project, code ANT22991 and by FAEPEX 2273/24. The first author was also supported by project PFR22-002 from VRIIP-UA}

\author{María Alejandra Alvarez}
\address{María Alejandra Alvarez\\
Departamento de Matemáticas, Facultad de Ciencias Básicas, 
Universidad de Antofagasta, Antofagasta, Chile}
\email{maria.alvarez@uantof.cl (María Alejandra Alvarez)}

\author{Artem Lopatin}
\address{Artem Lopatin\\
Universidade Estadual de Campinas (UNICAMP), Campinas, SP, Brazil}
\email{dr.artem.lopatin@gmail.com (Artem Lopatin)}

\begin{abstract}
We classify all two-dimensional simple algebras (which may be non-associative) over an algebraically closed field. For each two-dimensional algebra $\algA$,  we describe a minimal (with respect to inclusion) generating set for the algebra of invariants of the $m$-tuples of $\algA$ in the case of characteristic zero. In particular, we establish that for any two-dimensional simple algebra $\algA$ with a non-trivial automorphism group, the Artin--Procesi--Iltyakov Equality holds for $\algA^m$; that is, the algebra of polynomial invariants of $m$-tuples of $\algA$ is generated by operator traces. As a consequence, we describe two-dimensional algebras that admit a symmetric or skew-symmetric invariant nondegenerate bilinear form.



\end{abstract}

\maketitle

\tableofcontents

\section{Introduction}

\subsection{Algebra of invariants}\label{section_alg_inv} Assume that $\FF$ is an algebraically closed field of an arbitrary characteristic $\Char{\FF}$. All vector spaces and algebras are over $\FF$.   

Assume that $\algA$ is an $\FF$-algebra of dimension $n$, i.e., $\algA$ is a vector space equipped with a bilinear multiplication, which is not necessarily associative. Consider a subgroup $G$ of the group of all automorphisms $\Aut(\algA)\leqslant \GL_n$ of the algebra $\algA$. Given $m>0$, the group $G$ acts diagonally on $\algA^m = \algA \oplus \cdots \oplus \algA$ (with $m$ summands), i.e., $g\cdot\un{a} = (g\cdot a_1,\ldots,g\cdot a_m)$ for all $g\in G$ and $\un{a}=(a_1,\ldots,a_m)$ from $\algA^m$. The coordinate ring of the affine variety $\algA^m$ is a polynomial algebra (i.e., commutative and associative)  $$\FF[\algA^m] = \FF[x_{ri}\,|\, 1\leqslant r\leqslant m, \; 1\leqslant i\leqslant n].$$
Fix a basis $\{e_1,\ldots,e_n\}$ for $\algA$, and for any $a\in\algA$ denote by $(a)_i$ the $i^{\rm th}$ coordinate of $a$ with respect to the given basis. 
Note that we can consider $x_{ri}$ as a function $\algA^m\to\FF$ defined by $x_{ri}(\un{a}) = (a_r)_i$. Hence, elements of $\FF[\algA^m]$ can be interpreted as polynomial functions $\algA^m\to\FF$. 
The action of $G$ on $\algA^m$ induces the action on the coordinate ring $\FF[\algA^m]$ as follows: $(g\cdot f)(\un{a}) = f(g^{-1}\cdot\un{a})$ for all $g\in G$, $f\in \FF[\algA^m]$ and $\un{a}\in\algA^m$. The algebra of {\it $G$-invariants of the $m$-tuple of the algebra $\algA$} is
$$\FF[\algA^m]^{G}=\{f\in \FF[\algA^m]\,|\,g\cdot f=f \text{ for all }g\in G\},$$
or, equivalently, 
$$\FF[\algA^m]^{G}=\{f\in \FF[\algA^m]\,|\,f(g\cdot\un{a})=f(\un{a}) \text{ for all }g\in G,\; \un{a}\in \algA^m\}.$$

For short, the algebra of $\Aut(\algA)$-invariants of the $m$-tuple of $\algA$ is called the algebra of {\it invariants of the $m$-tuple of $\algA$}, and we denote it by
$$I_m(\algA):=\FF[\algA^m]^{\Aut(\algA)}.$$%
It is well known that the so-called operator traces (see Section~\ref{section_tr} for the details) are $G$-invariants. Denote by $\Tr(\algA)_m$ the subalgebra of $I_m(\algA)$ generated by all operator traces together with $1$. We say that the {\it Artin--Procesi--Iltyakov Equality} holds for $\algA^m$ if 
$$I_m(\algA) = \Tr(\algA)_m.$$%
It is known that for every $m\geqslant 1$ the Artin--Procesi--Iltyakov Equality holds for $\algA^m$ for the following algebras:
\begin{enumerate}
\item[$\bullet$] $\algA=M_n$ is the algebra of all $n\times n$ matrices over $\FF$, in case $\Char{\FF}=0$ or $\Char{\FF}>n$ (see, for example, Proposition~\ref{prop_M});

\item[$\bullet$] $\algA=\OO$ is the octonion algebra, in case $\Char{\FF}\neq2$ (see, for example,  Proposition~\ref{prop_O});

\item[$\bullet$] $\algA={\mathbb A}$ is the split Albert algebra, i.e., the exceptional simple Jordan algebra of $3\times 3$ Hermitian matrices over $\OO$ with the symmetric multiplication  $a\circ b = (ab+ba)/2$, in case $\Char{\FF}=0$ and $m\in\{1,2\}$.
\end{enumerate}

The three aforementioned algebras of invariants have been intensively studied.  M. Artin~\cite{Artin_1969} conjectured that the algebra of invariants  $I_m(M_n)$ is generated by traces of products of generic matrices. This conjecture was independently established by Sibirskii~\cite{Sibirskii68} and Procesi~\cite{Procesi76} in case $\Char\FF=0$. Later, Donkin~\cite{Donkin92a} described the generators for the algebra of invariants $I_m(M_n)$ when  $\Char\FF>0$. Minimal generating sets for the cases $n=2,3$ were obtained in~\cite{Procesi84, DKZ02,  Lopatin_Comm1, Lopatin_Sib, Lopatin_Comm2} and for $n=4,5$ with small values of $m$, in~\cite{Teranishi86, Drensky_Sadikova_4x4,Djokovic07}, assuming $\Char\FF=0$. The algebra of invariants $I_m(M_n)$ admits generalizations to invariants of representations of quivers and their various extensions (see~\cite{Lopatin09JA} for more details and references). 

A generating set for the algebra of invariants of the $m$-tuple of $\OO$ was constructed by Schwarz~\cite{schwarz1988} over the field of complex numbers $\CC$. This result was generalized to an arbitrary infinite field of odd characteristic by Zubkov and Shestakov in~\cite{zubkov2018}. Moreover, in case $\Char{\FF}\neq2$ a minimal generating set was constructed by Lopatin and Zubkov in~\cite{LZ_2}, using the classification of pairs of octonions from~\cite{LZ_1}.

Working over a field of characteristic zero, Iltyakov~\cite{Iltyakov_1995} proved that $I_m({\mathbb A})$ is integral over $\Tr({\mathbb A})_m$ and established the equality $I_m({\mathbb A}) = \Tr({\mathbb A})_m$ for $m\in\{1,2\}$.  Polikarpov~\cite{Polikarpov_1991} constructed a minimal generating set for $\Tr({\mathbb A})_2$ in case $\Char{\FF}\neq2,3$.

Assume from now on that $\Char{\FF}=0$ for the remainder of Section~\ref{section_alg_inv}.  Due to the following straightforward remark, the equality $I_m(\algA) = \Tr(\algA)_m$ does not hold in general. 

\begin{remark}\label{remark_trivial} For an $n$-dimensional algebra $\algA$ and $m>0$, the following two conditions are equivalent:
\begin{enumerate}
\item[(a)] $I_m(\algA)=\FF[\algA^m]$,

\item[(b)] the group $\Aut(\algA)$ is trivial.
\end{enumerate}%
In both of these cases, the Artin--Procesi--Iltyakov Equality does not hold for $\algA^m$ in case $n\geqslant 3$, since the homogeneous component of $I_m(\algA)$ of degree one has dimension at most two.  
\end{remark}

\noindent{}Moreover, the equality $I_m(\algA) = \Tr(\algA)_m$ does not hold in general for simple algebras of dimension greater than two, since a simple $n$-dimensional algebra with the trivial automorphism group was constructed by L'vov and Martirosyan~\cite{Lvov_Martirosyan_1982} (see also~\cite{Popov_1995}) for every $n\geqslant2$.


If a finite-dimensional simple algebra $\algA$ is generated by $m$ elements, Iltyakov~\cite{Iltyakov_1995GD} established that the field of rational invariants $\FF(\algA^m)^{\Aut(A)}$ is equal to the field of fractions of $\Tr(\algA)_m$, where $\FF(\algA^m)$ denotes the field of rational functions in the variables $\{x_{ri}\,|\,1\leqslant r\leqslant m, \; 1\leqslant i\leqslant n\}$. Iltyakov and Shestakov~\cite{Iltyakov_Sh_1996} also proved that the field $\FF({\mathbb A}^m)^{\Aut({\mathbb A})}$ is rational. For a finite-dimensional simple algebra $\algA$ over the  field $\FF=\CC$ of complex numbers, Iltyakov~\cite{Iltyakov_1998} described generators for $I_m(\algA)$ in terms of Laplace operators, in case where there exists an associative symmetric $\Aut(\algA)$-invariant nondegenerate bilinear form $\varphi:\algA^2\to \FF$ and $\algA$ has a compact real form on which $\varphi$ is positive definite (see Section~\ref{section_cor} for the definitions).  Elduque and Iltyakov~\cite{Iltyakov_Elduque_1999} proved that the algebra $\Tr(\algA)_m$ is finitely generated for every finite-dimensional algebra $\algA$. 

Note that there are other polynomial invariants for $n$-dimensional algebras, which arise  when we consider an $n$-dimensional algebra as an element of $\algW=\algV^{\ast}\otimes \algV^{\ast} \otimes \algV$, where $\algV=\FF^n$, and $\GL_n$ acts on $\algW$ via algebra isomorphisms (see~\cite{Munoz_Rosado_2012, Popov_2016} for more details).

\subsection{Results}\label{section_results} 

Working over a field of arbitrary characteristic, we classify all two-dimensional simple algebras in Theorem~\ref{theo_simple}. As a consequence, 
we prove that the automorphism group of a two-dimensional simple algebra is finite (see Corollary~\ref{cor_simple}). In case $\Char{\FF}=0$, for each two-dimensional algebra $\algA$, we describe a minimal (with respect to inclusion) generating set for the algebra of invariants of the $m$-tuples of $\algA$ (see Theorem~\ref{theo_gens}). In particular, given a two-dimensional simple algebra $\algA$ with a non-trivial automorphism group,
the Artin--Procesi--Iltyakov Equality holds for $\algA^m$ (see Corollary~\ref{cor_main}). We also explicitly described a series of {\it two-dimensional} simple algebras with the trivial automorphisms groups such that the Artin--Procesi--Iltyakov Equality does not hold for $\algA^m$ (see Lemma~\ref{lemma_non_API}). As a consequence of Theorem~\ref{theo_gens}, we characterize two-dimensional algebras, which can be endowed with a symmetric or skew-symmetric  invariant nondegenerate bilinear form (see Proposition~\ref{prop_form}). In particular, we show that any two-dimensional algebra with an infinite automorphism group does not admit a  symmetric invariant nondegenerate bilinear form (see Corollary~\ref{cor_form}). 

In Section~\ref{section_prelim}, we provide key definitions of generic elements and operator traces, along with some notations. In Section~\ref{section_trace_alg}, we prove Proposition~\ref{prop_trace}, which gives an explicit formula for calculating operator traces. As an example, in Section~\ref{section_example}, we consider matrix invariants and invariants of octonions. In Section~\ref{section_gen}, we present classical methods for calculating generators for an algebra of invariants in the characteristic zero case: reduction to the multilinear case, Weyl's polarization theorem, and Noether's theorem for finite groups. The classification of two-dimensional algebras, obtained by Kaygorodov and Volkov~\cite{Kaygorodov_Volkov_2019}, is presented in Section~\ref{section_simple2dim}. Using this classification, we describe all two-dimensional simple algebras in Theorem~\ref{theo_simple}. Finally, in Section~\ref{section_invariants}, we prove our main results: Theorem~\ref{theo_gens}, Theorem~\ref{theo_API}, and Corollary~\ref{cor_main}. An application to bilinear forms is considered in Section~\ref{section_cor}.

To the best of our knowledge, there is no known counterexample to the following conjecture: over a field of characteristic zero, the Artin--Procesi--Iltyakov Equality holds for $\algA^m$ for every simple algebra $\algA$ with a non-trivial automorphism group.

\begin{problem}\label{problem_relations}
Over a field of characteristic zero, describe the relations between generators of the algebra of invariants $I_m(\algA)$ for every two-dimensional algebra $\algA$.
\end{problem}

Since in many cases the algebra of invariants $I_m(\algA)$ is generated by operator traces (see Theorem~\ref{theo_API}), Problem~\ref{problem_relations} is closely connected with the problem of describing polynomial identities for two-dimensional algebras. The systematic study of the latter problem was initiated in~\cite{PIs_Jordan_dim2_2023} and continued in~\cite{PIs_Novikov_dim2_2025}, where the cases of Jordan algebras and Novikov algebras, respectively, were considered.

\subsection{Notations}\label{section_notations} 

A monomial $w= x_{r_1,i_1}\cdots x_{r_k,i_k}$ from $\FF[\algA^m]$ has {\it multidegree} $\mdeg(w)=(\de_1,\ldots,\de_m)\in\NN^m$, where $\de_r$ is the number of letters of $w$ lying in the set $\{x_{r1},\ldots, x_{rn}\}$ and $\NN=\{0,1,2,\ldots\}$. As an example, $\mdeg(x_{21}x_{32}x_{22})=(0,2,1)$ for $m=3$. For short, we denote the multidegree $(1,\ldots,1)\in\NN^m$ by $1^m$. If $f\in\FF[\algA^m]$ is a linear combination of monomials of the same multidegree $\un{d}$, then we say that $f$ is $\NN^m$-{\it homogeneous} of multidegree $\un{d}$. In other words, we have defined the $\NN^m$-grading of $\FF[\algA^m]$ by multidegrees. Since $\FF$ is infinite, the algebra of invariants $\FF[\algA^m]^G$ also has the $\NN^m$-grading by multidegrees. An $\NN$-homogeneous element $f\in\FF[\algA^m]$ of multidegree $(d_1,\ldots,d_m)$ with $d_1,\ldots,d_m\in\{0,1\}$ is called {\it multilinear}. For short, we denote by  $x_{r_1,i_1}\cdots\widehat{x_{r_l,i_l}}\cdots x_{r_k,i_k}$ the monomial $x_{r_1,i_1}\cdots x_{r_{l-1},i_{l-1}} x_{r_{l+1},i_{l+1}} \cdots x_{r_k,i_k}$ from   $\FF[\algA^m]$, where $1\leqslant l\leqslant k$. Given a subset $S\subset\algA$, we denote by $\alg\{S\}$ the subalgebra of $\algA$  (without unity in general) generated by $S$.

Given a vector $\un{r}=(r_1,\ldots,r_k)$, we write $\#\un{r}=k$ and $|\un{r}|=r_1+\cdots + r_k$. Denote $\FF^{\times}=\FF\backslash\{0\}$.

\section{Preliminaries on invariants}\label{section_prelim}

\subsection{Generic elements} To explicitly define the action of $G$ on $\FF[\algA^m]$ consider  the algebra $\widehat{\algA}_m=\algA \otimes_{\FF} \FF[\algA^m]$, which is a $G$-module, where the multiplication and the $G$-action are defined as follows: $(a\otimes f)(b\otimes h)=ab \otimes fh$ and $g\bullet (a\otimes f) = g \cdot a \otimes f$ for all $g\in G$, $a,b\in\algA$ and $f,h\in\FF[\algA^m]$. Define by 
$$X_r = \left(\begin{array}{c}
x_{r1} \\
\vdots \\
x_{rn} \\
\end{array}\right):=e_1 \otimes x_{r1} + \cdots + e_n \otimes x_{rn}$$
the {\it generic elements} of $\widehat{\algA}_m$, where $1\leqslant r\leqslant m$. In particular, we have 
$$g\bullet X_r = 
\left(\begin{array}{c}
g\cdot x_{r1} \\
\vdots \\
g\cdot x_{rn} \\
\end{array}\right) 
\in\widehat{\algA}_m.$$%
For $g\in G\leqslant \GL_n$ we write $[g]$ for the corresponding $n\times n$ matrix and $(g)_{ij}$ for the $(i,j)^{\rm th}$ entry of $[g]$. By straightforward calculations we can see that 
\begin{eq}\label{eq_action}
g\bullet X_r = [g]^{-1}X_r.
\end{eq}%

Denote by $\algX{m}$ the subalgebra of $\widehat{\algA}_m$  generated by the generic elements $X_1,\ldots,X_m$ and the unity $1$. Any product of the generic elements is called a word of $\algX{m}$. Since $G\leqslant \Aut(\algA)$, we  have that 
\begin{eq}\label{eq_action_prod}
g\bullet (FH) = (g\bullet F)(g\bullet H)
\end{eq}%
for all $F,H$ from $\algX{m}$.

\subsection{Operator traces}\label{section_tr}




For $a\in\algA$ denote by $L_a:\algA \to\algA$ and  $R_a:\algA \to\algA$  the operators of the left and right multiplication by $a$, respectively. Then define the {\it left operator trace} $\tr_{\rm L}:\algA \to \FF$ by $\tr_{\rm L}(a)=\tr(L_a)$ and the {\it right operator trace} $\tr_{\rm R}:\algA \to \FF$ by $\tr_{\rm R}(a)=\tr(R_a)$. 

To expand these constructions, we denote by $\FF\LA \chi_0,\ldots,\chi_m\RA$ the absolutely free unital algebra in letters $\chi_0,\chi_1,\ldots,\chi_m$ and for $f\in \FF\LA \chi_0,\ldots,\chi_m\RA$, $\un{a}=(a_0,\ldots,a_m)\in\algA^{m+1}$ define $f(\un{a})\in\algA$ as the result of substitutions $\chi_0\to a_0,\, \chi_1\to a_1,\ldots,\chi_m\to a_m$ in $f$. For $h\in\FF\LA \chi_0,\ldots,\chi_m\RA$ denote by $L_h,R_h:\FF\LA \chi_0,\ldots,\chi_m\RA \to \FF\LA \chi_0,\ldots,\chi_m\RA$ the operators of the left and right multiplication by $h$, respectively.   As usually, the composition of maps $P_1,P_2:\algA\to \algA$ is denoted by $P_1\circ P_2 (a)=P_1(P_2(a))$, $a\in\algA$. Similar notation we use for composition of maps $P_1,P_2:\FF\LA \chi_0,\ldots,\chi_m\RA \to \FF\LA \chi_0,\ldots,\chi_m\RA$. For a symbol $P\in\{L,R\}$, we write 
\begin{enumerate}
\item[$\bullet$] $P_a:\algA\to\algA$ for the operator of left or right multiplication by $a\in\algA$;

\item[$\bullet$] $P_h:\FF\LA \chi_0,\ldots,\chi_m\RA \to \FF\LA \chi_0,\ldots,\chi_m\RA$ for the operator of left or right multiplication by $h\in\FF\LA \chi_0,\ldots,\chi_m\RA$.
\end{enumerate}

\noindent{}The following definition of the operator trace generalizes definitions of the left and right operator traces.

\begin{defin}\label{def1}
Assume that $m\geqslant1$ and $h\in \FF\LA \chi_0,\ldots,\chi_m\RA$ is homogeneous of degree 1 in $\chi_{0}$, i.e., each monomial of $h$ contains $\chi_{0}$ exactly once. Then 
\begin{enumerate}
\item[$\bullet$] for every $\un{a}\in\algA^m$ define the linear operator $h(\,\cdot\,,\un{a}):\algA\to\algA$  by the following equality: $h(\,\cdot\,,\un{a})(b) = h(b,a_1,\ldots, a_m)$ for all $b\in\algA$;

\item[$\bullet$]  define {\it operator trace} as follows: $\tr(h):\algA^m \to \FF$, where $\tr(h)(\un{a})=\tr(h(\,\cdot\,,\un{a}))$ for all $\un{a}\in\algA^m$.
\end{enumerate}
\end{defin}

As an example, $\tr(\chi_0)=n$, $\tr_{\rm L}=\tr(\chi_1\chi_0)$, and $\tr_{\rm R}=\tr(\chi_0\chi_1)$. 

\begin{remark}\label{remark_key} Assume that $h\in \FF\LA \chi_0,\ldots,\chi_m\RA$ is homogeneous of degree 1 in $\chi_{0}$.
\begin{enumerate}
\item[(a)] It is easy to see that $\tr(h) \in \FF[\algA^m]$.

\item[(b)] It is well known that $\tr(h)$ is an invariant from $I_m(\algA)$ (for example, see Lemma~\ref{lemma_Tr_inv}). 

\item[(c)] If $h$ has multidegree $(1,\de_1,\ldots,\de_m)\in\NN^{m+1}$, then $\tr(h)$ is $\NN^m$-homogeneous of multidegree $(\de_1,\ldots,\de_m)$.
\end{enumerate}
\end{remark}

\begin{defin}\label{def3} Denote by $\tr(\algA)_m$ the subalgebra of $I_m(\algA)$ generated by $1$ and operator traces $\tr(h)$ for all that $h\in \FF\LA \chi_0,\ldots,\chi_m\RA$ homogeneous of degree 1 in $\chi_{0}$.
\end{defin}

Operator traces can also be defined in the following way: 

\begin{notation}\label{notation2}
\begin{enumerate}
\item[(a)] Assume that $P^1,\ldots,P^k\in\{L,R\}$ are symbols, $h_1,\ldots,h_k\in \FF\LA \chi_1,\ldots,\chi_m\RA$ and $k>0$. Define a {\it polynomial map} $\tr\!\big(P^1_{h_1} \circ \cdots \circ P^k_{h_k}\big):\algA^m \to \FF$ by 
$$\un{a} \to \tr\!\big(P^1_{h_1(\un{a})} \circ \cdots \circ P^k_{h_k(\un{a})}\big) \text{ for all }\un{a}\in\algA^m.$$

\item[(b)] We write $\tr(P)$ for a linear combination $P$ of compositions 
$\{P^1_{h_1} \circ \cdots \circ P^k_{h_k}\}$ as in item (a). 
\end{enumerate}
\end{notation}

\begin{remark}\label{remark_equiv}
Given $h= P^1_{h_1} \circ \cdots \circ P^k_{h_k} (\chi_0) \in  \FF\LA \chi_0,\ldots,\chi_m\RA$ for some symbols 
$P^1,\ldots,P^k\in\{L,R\}$,
$h_1,\ldots,h_k\in \FF\LA \chi_1,\ldots,\chi_m\RA$ and $k>0$,  we have
$$\tr(h)=\tr(P^1_{h_1} \circ \cdots \circ P^k_{h_k}).$$
\end{remark}

The third way to define some operator traces is the following one: 

\begin{notation}\label{notation3}
For a symbol $P\in\{L,R\}$ the linear map $\tr_{P}:\algA\to\FF$ can be extended to the map $\tr_{P}:\widehat{\algA}_m\to\FF[\algA^m]$ by the linearity.
\end{notation}

\begin{lemma}\label{lemma3}
For every $h\in \FF\LA \chi_1,\ldots,\chi_m\RA$ we have
$$\tr\nolimits_{\rm L}(h(X_1,\ldots,X_m)) = \tr(h(\chi_1,\ldots,\chi_m)\chi_0)
\;\; \text{ and } \;\;
\tr\nolimits_{\rm R}(h(X_1,\ldots,X_m)) = \tr(\chi_0 h(\chi_1,\ldots,\chi_m)).$$
\end{lemma}
\begin{proof} By the linearity, we can assume that $h$ is a monomial. Consider some $\un{a}=(a_1,\ldots,a_m)\in\algA^m$. By the definition of the product in $\widehat{\algA}_m$, we have 
$\tr_{\rm L}(h(X_1,\ldots,X_m))(\un{a}) = \tr_{\rm L}(h(\un{a}))$. On the other hand, $\tr(h(\chi_1,\ldots,\chi_m)\chi_0)(\un{a})=\tr(L_{h(\un{a})})$. The first claim of the lemma is proved. The second equality can be proved similarly.
\end{proof}

\section{Traces for algebras}\label{section_trace_alg}

\begin{lemma}\label{lemma_Tr_inv}
Assume $h\in \FF\LA \chi_0,\ldots,\chi_m\RA$ is homogeneous of degree 1 in $\chi_{0}$. Then the operator trace $\tr(h)$ lies in $I_m(\algA)$.
\end{lemma}
\begin{proof} The statement of the lemma is equivalent to the following claim:
\begin{eq}\label{eq_claim1}
\tr(h)(g\cdot\un{a}) = \tr(h)(\un{a})
\end{eq}%
for every $\un{a}\in \algA^m$ and $g\in \Aut(\algA)$. We have 
$$\tr(h)(\un{a}) = \tr(h(\,\cdot\,,\un{a})) = \sum\limits_{i=1}^n \big(h(e_i,\un{a})\big)_i 
\quad \text{ and }\quad
\tr(h)(g\cdot\un{a}) =  \sum\limits_{i=1}^n \big(h(e_i,g\cdot\un{a})\big)_i.$$
Since $g$ is an automorphism of $\algA$, we have $\big(h(e_i,g\cdot\un{a})\big)_i = \big( g\cdot h(g^{-1}\cdot e_i, \un{a}) \big)_i=\sum_{j=1}^n (g)_{ij} \big(h(g^{-1}\cdot e_i,g\cdot\un{a})\big)_j$. Thus, 
$$\tr(h)(g\cdot\un{a}) = \sum_{j,k=1}^n \Big( \sum_{i=1}^n (g^{-1})_{ki} (g)_{ij} \Big) \big(h(e_k,g\cdot\un{a})\big)_j = \sum_{j=1}^{n} \big(h(e_j,g\cdot\un{a})\big)_j.$$
Claim~(\ref{eq_claim1}) is proven.
\end{proof}

Assume that the tableau of multiplication of the algebra $\algA$ is $M=(M_{ij})_{1\leqslant i,j\leqslant n}$, i.e., $e_ie_j=M_{ij}$ in $\algA$. Denote
$$M_{ij}=\sum_{l=1}^n M_{ijl} e_l$$
for some $ M_{ijl}\in\FF$. For every $1\leqslant i\leqslant n$ consider the following $n\times n$ matrices over $\FF$: 
$$\ML{i}=(M_{ijl})_{1\leqslant l,j\leqslant n} \;\; \text{ and  }\;\;
\MR{i}=(M_{jil})_{1\leqslant l,j\leqslant n}.$$%

\begin{prop}\label{prop_trace}
Assume that $m\geqslant1$ and a homogeneous monomial $h\in \FF\LA \chi_0,\ldots,\chi_m\RA$ of degree 1 in $\chi_{0}$ satisfies the equality $h= P^1_{\chi_{r_1}} \circ \cdots \circ P^k_{\chi_{r_k}} (\chi_0)$ for some symbols $P^1,\ldots,P^k\in\{L,R\}$, $1\leqslant r_1,\ldots,r_k\leqslant m$ and $k>0$. Then 

$$\tr(h)=\sum_{1\leqslant i_1,\ldots,i_k\leqslant n} \tr\!\Big(M_{P^1}^{(i_1)} \cdots M_{P^k}^{(i_k)}\Big) x_{r_1,i_1} \cdots x_{r_k,i_k}.$$
\end{prop}
\begin{proof} Assume $a=\al_1 e_1+\cdots+\al_n e_n$ for some $\al_1,\ldots,\al_n\in\FF$. Denote 
$$L_a = ((L_a)_{ij})_{1\leqslant i,j\leqslant n} \;\; \text{ and } \;\; 
R_a = ((R_a)_{ij})_{1\leqslant i,j\leqslant n}, $$
where $(L_a)_{ij},(R_a)_{ij}\in\FF$. Then for $1\leqslant j\leqslant n$ we have
$$a e_j = \sum_{l=1}^n\sum_{i=1}^n \al_i M_{ijl} e_l  \;\; \text{ and } \;\;  
e_ja =  \sum_{l=1}^n\sum_{i=1}^n \al_i M_{jil} e_l .$$
Hence, 
\begin{eq}\label{eq_tr1}
(L_a)_{lj}=\sum_{i=1}^n \al_i M_{ijl} \;\; \text{ and } \;\; 
(R_a)_{lj}=\sum_{i=1}^n \al_i M_{jil}.
\end{eq}
\noindent{}Therefore, for every symbol $P\in\{L,R\}$ we have
\begin{eq}\label{eq_tr2}
(P_a)_{lj}=\sum_{i=1}^n \al_i (M_P^{(i)})_{lj}.
\end{eq}

Consider $\un{a}=(a_1,\ldots, a_m)\in\algA$ and let $a_s=\al_{s1}e_1 + \cdots +\al_{sn}e_n$ for some $\al_{s1},\ldots,\al_{sn}\in\FF$, where $1\leqslant s\leqslant m$. Then 
$$\begin{array}{rcl}
\tr(h)(\un{a}) &= &\tr(P^1_{a_{r_1}}\circ\cdots \circ P^k_{a_{r_k}}) \\
 &= &
\sum\limits_{1\leqslant j_1,\ldots,j_k\leqslant n} 
\Big(P^1_{a_{r_1}}\Big)_{j_1j_2}  \Big(P^2_{a_{r_1}}\Big)_{j_2j_3} 
 \cdots   \Big(P^k_{a_{r_k}}\Big)_{j_kj_1}. \\
\end{array}
$$
\noindent{}Applying formula~(\ref{eq_tr2}) we obtain
$$\begin{array}{rcl}
\tr(h)(\un{a}) &= &\sum\limits_{1\leqslant j_1,\ldots,j_k\leqslant n} \;\;
\sum\limits_{1\leqslant i_1,\ldots,i_k\leqslant n} 
\al_{r_1,i_1} \cdots \al_{r_k,i_k} 
\Big(M_{P^1}^{(i_1)}\Big)_{j_1j_2}  \Big(M_{P^2}^{(i_2)}\Big)_{j_2j_3} 
 \cdots   \Big(M_{P^k}^{(i_k)}\Big)_{j_kj_1} \\
 & = & \sum\limits_{1\leqslant i_1,\ldots,i_k\leqslant n} \tr\!\Big(M_{P^1}^{(i_1)} \cdots M_{P^k}^{(i_k)}\Big) \al_{r_1,i_1} \cdots \al_{r_k,i_k}.
\end{array}
$$
\noindent{}The claim of the proposition is proven.
\end{proof}

Proposition~\ref{prop_trace} (or, equivalently, see equalities~(\ref{eq_tr1})) implies that for all $1\leqslant r\leqslant m$ we have 
\begin{eq}\label{eq_trL_trR}
\tr(\chi_r\chi_0)=   \sum_{i,j=1}^n x_{ri} M_{ijj}
\;\; \text{ and } \;\; 
\tr(\chi_0\chi_r)= \sum_{i,j=1}^n x_{ri} M_{jij}.
\end{eq}

For every $1\leqslant i,i'\leqslant n$ consider the following $n\times n$ matrices over $\FF$: 
$$\MdL{i}{i'}=\Big( \sum_{t=1}^nM_{i i' t} M_{tjl} \Big)_{1\leqslant l,j\leqslant n} \;\; \text{ and  }\;\;
\MdR{i}{i'} = \Big( \sum_{t=1}^nM_{i i' t} M_{jtl} \Big)_{1\leqslant l,j\leqslant n}.$$

\begin{prop}\label{prop_trace_double}
Assume that $m\geqslant1$ and a homogeneous monomial $h\in \FF\LA \chi_0,\ldots,\chi_m\RA$ of degree 1 in $\chi_{0}$ satisfies the equality $h= P^1_{h_1} \circ \cdots \circ P^k_{h_k} (\chi_0)$ for some symbols $P^1,\ldots,P^k\in\{L,R\}$ and monomials $h_1,\ldots,h_k\in \FF\LA \chi_1,\ldots,\chi_m\RA$ of degree 1 or 2, where $k>0$. For every $1\leqslant q\leqslant k$ denote 
\begin{enumerate}
\item[$\bullet$] $\un{i}_q=i_q$ and $w(q,\un{i}_q) = x_{r_q,i_q}$, in case $h_q=\chi_{r_q}$ for some 
$1\leqslant r_q\leqslant m$;

\item[$\bullet$] $\un{i}_q=\{i_q,i'_q\}$, $M_{P^q}^{(\un{i}_q)}=M_{P^q}^{(i_q,i'_q)}$ and $w(q,\un{i}_q) = x_{r_q,i_q}x_{r'_q,i'_q}$, in case $h_q=\chi_{r_q}\chi_{r'_q}$ for some 
$1\leqslant r_q,r'_q\leqslant m$.
\end{enumerate}

Then 

$$\tr(h)=\sum_{1\leqslant \un{i}_1,\ldots,\un{i}_k\leqslant n} \tr\!\Big(M_{P^1}^{(\un{i}_1)} \cdots M_{P^k}^{(\un{i}_k)}\Big) w(1,\un{i}_1) \cdots w(k,\un{i}_k),$$
\noindent{}where the condition $1\leqslant\{i_q,i'_q\}\leqslant n$ stands for the condition $1\leqslant i_q,i'_q \leqslant n$.
\end{prop}
\begin{proof} Assume $a=\al_{1} e_1+\cdots+\al_{n} e_n$ and $a'=\al'_{1} e_1+\cdots+\al'_{n} e_n$ for some $\al_{1},\ldots,\al_{n}$, $\al'_{1},\ldots,\al'_{n}$ from $\FF$. Denote 
$$L_{aa'} = ((L_{aa'})_{ij})_{1\leqslant i,j\leqslant n} \;\; \text{ and } \;\; 
R_{aa'} = ((R_{aa'})_{ij})_{1\leqslant i,j\leqslant n}, $$
where $(L_{aa'})_{ij},(R_{aa'})_{ij}\in\FF$. Then for $1\leqslant j\leqslant n$ we have
$$(aa') e_j = \sum_{1\leqslant i,i'\leqslant n}\al_{i}\al'_{i'} (e_{i}e_{i'})e_j =
\sum_{1\leqslant i,i',t\leqslant n} \al_{i}\al'_{i'} M_{ii't}\, e_t e_j=
\sum_{1\leqslant i,i',t,l\leqslant n} \al_{i}\al'_{i'} M_{ii't}\, M_{tjl}\, e_l,$$
$$e_j (aa')  = \sum_{1\leqslant i,i'\leqslant n}\al_{i}\al'_{i'} e_j(e_{i}e_{i'}) =
\sum_{1\leqslant i,i',t\leqslant n} \al_{i}\al'_{i'} M_{ii't}\, e_je_t =
\sum_{1\leqslant i,i',t,l\leqslant n} \al_{i}\al'_{i'} M_{ii't}\, M_{jtl}\, e_l.$$

Hence, 
\begin{eq}\label{eq_tr1_double}
(L_{aa'})_{lj}=\sum_{1\leqslant i,i',t\leqslant n}  \al_{i}\al'_{i'} M_{ii't}\, M_{tjl}
\;\; \text{ and } \;\; 
(R_{aa'})_{lj}=\sum_{1\leqslant i,i',t\leqslant n} \al_{i}\al'_{i'} M_{ii't}\, M_{jtl}
\end{eq}
\noindent{}Therefore, for every symbol $P\in\{L,R\}$ we have
\begin{eq}\label{eq_tr2_double}
(P_{aa'})_{lj}=\sum_{1\leqslant i,i'\leqslant n} \al_{i}\al'_{i'} \Big(M_{P}^{(i,i')}\Big)_{lj}.
\end{eq}

Consider $\un{a}=(a_1,\ldots, a_m)\in\algA$ and let $a_s=\al_{s1}e_1 + \cdots +\al_{sn}e_n$ for some $\al_{s1},\ldots,\al_{sn}\in\FF$, where $1\leqslant s\leqslant m$. Then 
$$\begin{array}{rcl}
\tr(h)(\un{a}) &= &\tr\big(P^1_{h_1(\un{a})}\circ\cdots \circ P^k_{h_k(\un{a})}\big) \\
 &= &
\sum\limits_{1\leqslant j_1,\ldots,j_k\leqslant n} 
\Big(P^1_{h_1(\un{a})}\Big)_{j_1j_2}  \Big(P^2_{h_2(\un{a})}\Big)_{j_2j_3} 
 \cdots   \Big(P^k_{h_k(\un{a})}\Big)_{j_kj_1} \\
\end{array}
$$%

\noindent{}For every $1\leqslant q\leqslant k$ denote 
\begin{enumerate}
\item[$\bullet$] $\al(q,\un{i}_q) = \al_{r_q,i_q}$, in case $h_q=\chi_{r_q}$ for some 
$1\leqslant r_q\leqslant m$;

\item[$\bullet$] $\al(q,\un{i}_q) = \al_{r_q,i_q} \al_{r'_q,i'_q}$, in case $h_q=\chi_{r_q}\chi_{r'_q}$ for some 
$1\leqslant r_q,r'_q\leqslant m$.
\end{enumerate}%

\noindent{}Since $h_1,\ldots,h_k$ are monomials of degree 1 or 2, we apply formulas~(\ref{eq_tr2}) and~(\ref{eq_tr2_double}) to obtain
$$\begin{array}{rcl}
\tr(h)(\un{a}) &= &\sum\limits_{1\leqslant j_1,\ldots,j_k\leqslant n} \;\;
\sum\limits_{1\leqslant \un{i}_1,\ldots,\un{i}_k\leqslant n} 
\al(1,\un{i}_1) \cdots \al(k,\un{i}_k) 
\Big(M_{P^1}^{(\un{i}_1)}\Big)_{j_1j_2}  \Big(M_{P^2}^{(\un{i}_2)}\Big)_{j_2j_3} 
 \cdots   \Big(M_{P^k}^{(\un{i}_k)}\Big)_{j_kj_1} \\
 & = & \sum\limits_{1\leqslant \un{i}_1,\ldots,\un{i}_k\leqslant n} \tr\!\Big(M_{P^1}^{(\un{i}_1)} \cdots M_{P^k}^{(\un{i}_k)}\Big) \al(1,\un{i}_1) \cdots \al(k,\un{i}_k) .
\end{array}
$$
\noindent{}The claim of the proposition is proven.
\end{proof}

\section{Examples}\label{section_example}

The results from this section are well known. We present complete proofs for the sake of completeness.

\subsection{Matrix invariants} \label{section_example_matrix}

The general linear group $\GL_n$ acts on the algebra $M_n=M_n(\FF)$ of  $n\times n$ matrices by conjugations: $g\cdot A = gAg^{-1}$ for all $A\in M_n$ and $g\in\GL_n$. It is well known that the group of automorphisms $\Aut(M_n)$ is $\PGL_n= \GL_n / \,\FF^{\times}$, where the action of $\PGL_n$ on $M_n$ is given by $\ov{g}\cdot A = gAg^{-1}$ for $g\in\GL_n$. Obviously, $\FF[M_n^m]^{\PGL_n} = \FF[M_n^m]^{\GL_n}$. 

For short, in this section  we denote $\LA i,i'\RA = (i-1)n + i'$. 
Consider the basis $\{e_j\,|\,1\leqslant j\leqslant n^2\}$ of $M_n$ defined by $E_{ii'}=e_{\LA i,i'\RA}$ for all $1\leqslant i,i'\leqslant n$. Since $\widehat{(M_n)}_m\simeq M_n(\FF[M_n^m])$, we can identify the generic element $X_r$ with the $n\times n$ matrix $(x_{r,\LA i,i'\RA})_{1\leqslant i,i'\leqslant n}$, i.e., the product of the generic elements in $\widehat{(M_n)}_m$ is the product of the corresponding matrices.

\begin{prop}\label{prop_M}
The Artin--Procesi--Iltyakov Equality holds for $M_n^m$ for all $m>0$, in case $\Char{\FF}=0$ or $\Char{\FF}>n$. 
\end{prop}
\begin{proof}

The tableau of multiplication $M= (M_{ij})_{1\leqslant i,j\leqslant n}$ for $M_n$ is given by the following equalities: $e_{\LA i,i'\RA} e_{\LA j,j'\RA} = E_{ii'} E_{jj'}=\de_{i'j} e_{\LA i,j'\RA}$. Thus,

\begin{eq}\label{eq_MMM}M_{\LA i,i'\RA\LA j,j'\RA\LA j,j'\RA}=\left\{
\begin{array}{rl}
1,& i'=j \text{ and } i=j \\
0,& \text{otherwise} \\
\end{array}
\right.
\end{eq}

\noindent{}Lemma~\ref{lemma3} together with formulas~(\ref{eq_trL_trR}) and (\ref{eq_MMM}) imply that
\begin{eq}\label{eq_Tr_M}
\tr\nolimits_{\rm L}(X_r)=\tr(\chi_r \chi_0) = \sum_{i,i',j,j'=1}^n x_{r,ii'} M_{\LA i,i'\RA\LA j,j'\RA\LA j,j'\RA} = \sum_{i,j'=1}^n x_{r,ii}=n (x_{r,11} + \cdots + x_{r,nn}) = n \tr(X_r),
\end{eq}%
\noindent{}where the map $\tr_{L}:M_n(\FF[M_n^m])\to\FF[M_n^m]$ was defined in Notation~\ref{notation3} and $\tr$ stands for the usual trace of an $n\times n$ matrix over a commutative associative ring. Consider $1\leqslant r_1,\ldots,r_k\leqslant m$. It follows from formula~(\ref{eq_Tr_M}) that 
\begin{eq}\label{eq_trM}
\tr\nolimits_{\rm L}(X_{r_1}\cdots X_{r_k}) = n \tr(X_{r_1}\cdots X_{r_k}).
\end{eq}%

\noindent{}Since Lemma~\ref{lemma3} implies 
$$\tr(w(\chi_1,\ldots,\chi_m)\chi_0) = \tr\nolimits_{\rm L}(X_{r_1}\cdots X_{r_k})$$
for $w = (\cdots(\chi_{r_1}\chi_{r_2})\cdots )\chi_{r_{m}}$ from $\FF\LA \chi_1,\ldots,\chi_m\RA$, 
it follows from formula~(\ref{eq_trM}) that $\Tr(\algA)_m$ contains all elements $\tr(X_{r_1}\cdots X_{r_k})$ for $1\leqslant r_1,\ldots,r_k\leqslant m$.

Since the algebra $\FF[M_n^m]^{\GL_n}$ is known to be generated by $\tr(X_{r_1}\cdots X_{r_k})$ for $1\leqslant r_1,\ldots,r_k\leqslant m$ (see~\cite{Procesi76} in case of $\Char{\FF}=0$ and \cite{Donkin92a} in case of $\Char{\FF}>n$), the required claim is proven.
\end{proof}

\subsection{Invariants of octonions}\label{section_example_octonion}

The {\it octonion algebra} $\OO=\OO(\FF)$, also known as the {\it split Cayley algebra}, is the vector space of all matrices

$$a=\matr{\al}{\uu}{\vv}{\be}\text{ with }\al,\be\in\FF \text{ and } \uu,\vv\in\FF^3,$$%
endowed with the following multiplication:
$$a a'  =
\matr{\al\al'+ \uu\cdot \vv'}{\al \uu' + \be'\uu - \vv\times \vv'}{\al'\vv +\be\vv' + \uu\times \uu'}{\be\be' + \vv\cdot\uu'},\text{ where } a'=\matr{\al'}{\uu'}{\vv'}{\be'},$$%
$\uu\cdot \vv = u_1v_1 + u_2v_2 + u_3v_3$ and $\uu\times \vv = (u_2v_3-u_3v_2, u_3v_1-u_1v_3, u_1v_2 - u_2v_1)$. For short, denote  $\cc_1=(1,0,0)$, $\cc_2=(0,1,0)$,  $\cc_3=(0,0,1)$, $\zero=(0,0,0)$ from $\FF^3$. Consider the following basis of $\OO$: $$e_1=\matr{1}{\zero}{\zero}{0},\; e_8=\matr{0}{\zero}{\zero}{1},\; e_{i+1}=\matr{0}{\cc_i}{\zero}{0},\;e_{i+4}=\matr{0}{\zero}{\cc_i}{0}$$
for $i=1,2,3$.  Define the linear function {\it trace} by $\tr(a) = \al+\be$.

The group of all automorphisms of the algebra $\OO$ is the simple exceptional algebraic group $\G=\Aut(\OO)$.

Recall that the coordinate ring of $\OO^m$ is the polynomial $\FF$-algebra $\FF[\OO^m]=\FF[x_{ki}\,|\,1\leqslant k\leqslant m,\; 1\leqslant i\leqslant 8]$, where $x_{ri}:\OO^n\to\FF$ is defined by $(a_1,\ldots,a_m)\to \al_{ri}$ for 
\begin{eq}\label{eq7}
a_r=\matr{\al_{r1}}{(\al_{r2}, \al_{r3}, \al_{r4})}{(\al_{r5}, \al_{r6}, \al_{r7})}{\al_{r8}}\in\OO.
\end{eq}%

\noindent{}Since $\OO(\FF[\OO^m]) = \widehat{\OO}_m$, we can consider the generic element $X_r$ as
$$X_r= \matr{x_{r1}}{(x_{r2}, x_{r3}, x_{r4})}{(x_{r5}, x_{r6}, x_{r7})}{x_{r8}}\in\OO(\FF[\OO^m]).$$%
\noindent{}The trace $\tr:\OO\to\FF$ can be naturally extended to the linear function $\tr:\OO(\FF[\OO^m])\to \FF$.

\begin{prop}\label{prop_O}
The Artin--Procesi--Iltyakov Equality holds for $\OO^m$ for all $m>0$, in case $\Char{\FF}\neq2$. 
\end{prop}
\begin{proof} Consider the tableau of multiplication $M= (M_{ij})_{1\leqslant i,j\leqslant 8}$ for $\OO$: 
\begin{center}
\begin{tabular}{ c|cccccccc } 
 $M_{ij}$ & $e_1$ & $e_2$ & $e_3$ & $e_4$ & $e_5$ & $e_6$ & $e_7$ & $e_8$\\ 
  \hline
 $e_1$ &       $e_1$ & $e_2$ & $e_3$ & $e_4$ & $0$   & $0$   & $0$   & $0$ \\ 
 $e_2$ &       $0$   & $0$   & $e_7$ & $-e_6$& $e_1$ & $0$   & $0$   & $e_2$ \\ 
 $e_3$ &       $0$   & $-e_7$& $0$   & $e_5$ & $0$   & $e_1$ & $0$   & $e_3$ \\ 
 $e_4$ &       $0$   & $e_6$ & $-e_5$& $0$   & $0$   & $0$   & $e_1$ & $e_4$ \\ 
 $e_5$ &       $e_5$ & $e_8$ & $0$   & $0$   & $0$   & $-e_4$& $e_3$ &   $0$ \\ 
 $e_6$ &       $e_6$ & $0$   & $e_8$ & $0$ & $e_4$ & $0$   & $-e_2$&   $0$ \\ 
 $e_7$ &       $e_7$ & $0$   & $0$   & $e_8$ & $-e_3$ & $e_2$ & $0$ &   $0$ \\ 
 $e_8$ &         $0$ &   $0$ & $0$   & $0$   & $e_5$ & $e_6$ & $e_7$ & $e_8$ \\ 
\end{tabular}
\end{center}%
Thus, for $1\leqslant i,j\leqslant 8$ we have
\begin{eq}\label{eq_OOO}
M_{ijj}=\left\{
\begin{array}{rl}
1,& i=1 \text{ and } 1\leqslant j\leqslant 4 \\
1,& i=8 \text{ and } 5\leqslant j\leqslant 8 \\
0,& \text{otherwise} \\
\end{array}
\right.
\end{eq}

\noindent{}Lemma~\ref{lemma3} together with formulas~(\ref{eq_trL_trR}) and (\ref{eq_OOO}) imply that
\begin{eq}\label{eq_Tr_O}
\tr\nolimits_{\rm L}(X_r)=\tr(\chi_r \chi_0) = \sum_{i,j=1}^n x_{ri} M_{ijj} = 4 x_{r1} + 4x_{r8} = 4 \tr(X_r).
\end{eq}%
\noindent{}Consider $1\leqslant r_1,\ldots,r_k\leqslant m$. It follows from formula~(\ref{eq_Tr_O}) that 
\begin{eq}\label{eq_trO}
\tr\nolimits_{\rm L}((\cdots((X_{r_1}X_{r_2})X_{r_3})\cdots )X_{r_k}) = 4 \tr((\cdots((X_{r_1}X_{r_2})X_{r_3})\cdots )X_{r_k}).
\end{eq}%

\noindent{}Since Lemma~\ref{lemma3} implies 
$$\tr(w(\chi_1,\ldots,\chi_m)\chi_0) = \tr\nolimits_{\rm L}((\cdots((X_{r_1}X_{r_2})X_{r_3})\cdots )X_{r_k})$$
for $w = (\cdots((\chi_{r_1}\chi_{r_2})\chi_{r_3})\cdots )\chi_{r_{m}}$ from $\FF\LA \chi_1,\ldots,\chi_m\RA$, 
it follows from formula~(\ref{eq_trO}) that $\Tr(\algA)_m$ contains all elements $\tr(X_{r_1}\cdots X_{r_k})$ for $1\leqslant r_1,\ldots,r_k\leqslant m$.

Since the algebra $\FF[\OO^m]^{\G}$ is known to be generated by $\tr((\cdots((X_{r_1}X_{r_2})X_{r_3})\cdots )X_{r_k})$ for $1\leqslant r_1,\ldots,r_k\leqslant m$ (see~\cite{schwarz1988} in case of $\FF=\CC$ and \cite{zubkov2018} in case of $\Char{\FF}\neq2$), the required claim is proven.
\end{proof}

\section{Methods for calculation of generators for invariants}\label{section_gen}

In this section we assume that $G\leqslant  \GL_n$ and the characteristic of $\FF$ is zero. Given an $\NN$-graded algebra $\algB$, denote by $\beta(\algB)$ the least integer $\beta$ such that the algebra $\algB$ is generated by its $\NN$-homogeneous elements of degree $\leqslant\beta$. 

\subsection{Reduction to multilinear case}

A vector $\un{r}=(r_1,\ldots,r_m)\in\NN^m$ satisfying $r_1+\cdots +r_m=t$ is called a {\it partition} of $t\in \NN$ in $m$ parts. Denote by $\calcP^{t}_m$ the set of all such partitions. Given a partition $\un{r} \in \calcP^{t}_m$, we define the function $\un{r}|\cdot|:\{1,\ldots,t\}\to \{1,\ldots,m\}$ by
$$\un{r}|l|=j \text{ if and only if }r_1+\cdots+r_{j-1}+ 1\leqslant l \leqslant r_1+\cdots+r_{j},$$
for all $1\leqslant l\leqslant t$. For a partition $\un{r}\in \calcP^{t}_m $ we also define a homomorpism $\pi_{\un{r}}: \FF[\algA^t] \to \FF[\algA^m]$ of $\FF$-algebras by $x_{l,i}\to x_{\un{r}|l|,i}$ for all  $1\leqslant l\leqslant t$ and $1\leqslant i\leqslant n$.

The following proposition and lemma are well known and can easily be proven.

\begin{prop}\label{prop_char0}
The algebra of invariants $\FF[\algA^m]^G$ is generated by 
$$\{\pi_{\un{r}}(f)\,|\, f\in \FF[\algA^{t}]^G \text{ is multilinear},\; \un{r}\in \calcP^{t}_m,\; t\geqslant m\}.$$
\end{prop}

\begin{lemma}\label{lemma_diag}
If the group $G$ is diagonal, i.e., all elements of $G\leqslant  \GL_n$ are diagonal matrices, then the algebra of invariants $\FF[\algA^m]^G$ is generated by some monomials from $\FF[\algA^m]$.
\end{lemma}

\subsection{Polarization}  We consider the classical notion of polarization of an invariant as it was given  in~\cite{Draisma_Kemper_Wehlau_08}.

\begin{defin}\label{def_polar}  
Let $l,m\geqslant 1$ and $a_{r,s}$ be commutative indeterminates for all $1\leqslant r\leqslant l$, $1\leqslant s\leqslant m$. Define a homomorphism $\Phi=\Phi_{l,m}:\FF[\algA^l]\to \FF[\algA^m][a_{1,1},\ldots,a_{l,m}]$ of $\FF$-algebras by
$$\begin{array}{rcll}
\Phi(x_{r,i})&=&\sum\limits_{s=1}^m a_{r,s} x_{s,i} & (1\leqslant r\leqslant l,\; 1\leqslant i\leqslant n).\\
\end{array}
$$
\noindent{}Then for an $f\in \FF[\algA^l]$ there exists the set  $P=\Pol_{l}^m(f)\subset \FF[\algA^m]$ of non-zero elements such that $$\Phi(f) = \sum\limits_{h\in P} \un{a}^{\De(h)} h,$$
where  $\{\un{a}^{\De(h)}\}$ are pairwise different monomials in $\{a_{r,s}\}$. 
In other words, $\Pol_{l}^m(f)$ is the set of all non-zero coefficients of $\Phi(f)$, considered as a polynomial in indeterminates $\{a_{r,s}\}$.  For a set $S \subset \FF[\algA^l]$ we define $\Pol_{l}^m(S)$ to be the union of all $\Pol_{l}^m(f)$ with $f \in S$.
\end{defin}
\medskip

The next remark was proven in Section~1 of~\cite{Draisma_Kemper_Wehlau_08}.

\begin{remark}\label{remark_polar}
If $f\in \FF[\algA^l]^G$, then $\Pol_{l}^m(f)\subset \FF[\algA^m]^G$.
\end{remark}
\medskip

The following result was proven by Weyl~\cite{weyl1939classical}.
\begin{theo}[Weyl’s polarization theorem]\label{theo_Weyl}
Assume that the algebra of invariants $\FF[\algA^n]^G$ is generated by a set $S$. Then for every $m>n$ the algebra of invariants $\FF[\algA^m]^G$ is generated by  $\Pol_n^m(S)\subset \FF[\algA^m]^G$.
\end{theo}

\begin{remark}\label{rem_ex}
Assume that  $\algV$ is a finite dimensional vector space with $\dim \algV=n$, and $G\leqslant \GL(\algV)$. Note that in general, there is no upper bound on $\beta(\FF[\algV]^G)$, which only depends on $n$. As an example, assume $\FF=\CC$, $n=2$, $\FF[\algV]=\FF[x,y]$ and for any integer $q>1$ consider $\xi\in \CC$ with $\xi^q=1$ and $\xi^i\neq1$ for every $1\leqslant i<q$. Then for the group $G$ consisting of matrices 
$$\matr{\xi^i}{0}{0}{1},\quad 1\leqslant i\leqslant q,$$
we have that $\FF[\algV]^G=\FF[x^q,y]$ and $\beta(\FF[V]^G)=q$.
\end{remark}

\subsection{Invariant of finite groups}\label{section_finite_group}

In this section we assume that the group $G\leqslant\GL_n$ is finite. We will also use the following classical result by Emmy Noether~\cite{Noether_1916}:
\begin{theo}[Noether]\label{theo_Noether}
The algebra of invariants $\FF[\algA^m]^G$  is generated by $\NN^m$-homogeneous invariants of degree $\leqslant|G|$. In other words, $\beta(\FF[\algA^m]^G)\leqslant |G|$.
\end{theo}

Given a subgroup $H<G$ of a finite group $G$, consider
the {\it transfer map}
$$\Upsilon: \FF[\algA^m]^H \to \FF[\algA^m]^G, \quad \text{ } f\to \sum_{g\in G} g\cdot f.$$%

\noindent{}Obviously, $\Upsilon$ is a linear map. Moreover, $\Upsilon$ is surjective, since for $h\in \FF[\algA^m]^G$ we have $\Upsilon(\frac{1}{|G|}h) = h$ and $\frac{1}{|G|}h$ lies in $\FF[\algA^m]^H$. Therefore, the following remark holds.

\begin{remark}\label{remark_transfer} We use the above notations. Assume that $\De\in\NN^m$ is a multidegree and $S\subset \FF[\algA^m]^H$ is some subset of invariants of mutidegree $\De$ such that every invariant from $\FF[\algA^m]^H$ of multidegree $\De$ belongs to the $\FF$-span of $S$. Then every invariant from $\FF[\algA^m]^G$ of multidegree $\De$ belongs to the $\FF$-span of $\Upsilon(S)$.
\end{remark}

\section{Two-dimensional simple algebras}\label{section_simple2dim}

In this section we present the description of all two-dimensional simple algebras over an arbitrary algebraically closed field $\FF$. 

The classification of all two-dimensional algebras modulo the action of the group of automorphisms was considered in~\cite{Ananin_Mironov_2000, Petersson_2000, Goze_Remm_2011, Kaygorodov_Volkov_2019}.  See the introduction of~\cite{Kaygorodov_Volkov_2019} for the comparison of these results. Below in this section we present the results from Theorem 3.3 of~\cite{Kaygorodov_Volkov_2019} for the classification of 
algebras and Corollaries 3.8 and 4.2 of~\cite{Kaygorodov_Volkov_2019} for the description of its automorphisms. Note that the same results are also formulated in Table 1 of~\cite{Calderon_Ouridi_Kaygorodov_2022}. We apply the following notations.
\begin{enumerate}
\item[$\bullet$] $J=\matr{0}{1}{1}{0}$. 

\item[$\bullet$] Consider the action of the cyclic group 
$C_2 = \{1,\rho\}$ on $\FF$ defined
by the equality $\al^\rho=-\al$ for all $\al\in\FF$. We denote by $\FF_{\geqslant 0}$ some set of representatives of orbits under this action. As an example, if $\FF = \CC$, then we can take
$\CC_{\geqslant 0} = \{\al\in\CC \,|\, \re(\al) > 0\} \cup \{\al\in\CC \,|\, \re(\al) = 0, \im(\al) \geqslant 0\}$.

\item[$\bullet$] Similarly, consider the action of 
$C_2$ on $\FF^{\times}\backslash\{1\}$ defined
by the equality $\al^\rho={\al^{-1}}$ for all $\al\in\FF^{\times}\backslash\{1\}$. We denote by $\FF^{\times}_{>1}$ some set of representatives of orbits under this action.  As an example, if $\FF = \CC$, then we can take
$\CC^{\times}_{>1}= \{\al\in\CC^{\times} \,|\, |\al| > 1\} \cup \{\al\in\CC^{\times} \,|\, |\al| = 1,\; 0<{\rm arg}(\al)\leqslant \pi\}$. Note that $-1\in \FF^{\times}_{>1}$.



\item[$\bullet$] Consider the action of 
$C_2$ on $\FF^2$ defined
by the equality $(\al,\be)^\rho=(1-\al+\be,\be)$ for all $(\al,\be)\in\FF^2$. We denote by $\calU$ some set of representatives of orbits under this action.

\item[$\bullet$] Denote $\calT=\{(\al,\be)\in\FF^2\,|\,\al+\be=1\}$.

\item[$\bullet$] The definition of $\calV\subset \FF^4$ is given in Section 3 of~\cite{Kaygorodov_Volkov_2019}. Note that $(-1,-1,-1,-1)\in\calV$ if and only if $\Char(\FF)\neq3$.
\end{enumerate}

At first, we write the notation for the two-dimensional algebra $\algA$. Then we write a tableau of multiplication $M=(M_{ij})_{1\leqslant i,j\leqslant 2}$, i.e., $e_ie_j=M_{ij}$ in $\algA$. Finally, we write down each non-trivial elements $g=(g_{ij})_{1\leqslant i,j\leqslant 2}$ of $\Aut(\algA)\leqslant  \GL_2$. If $g$ is not written, then $\Aut(\algA)=\{{\rm Id}\}$. Note that in case $\algA=\bfE_1(-1,-1,-1,-1)$ the group $\Aut(\algA)\simeq \Sym_3$ is not explicitly described as a subgroup of $\GL_2$ in \cite{Kaygorodov_Volkov_2019, Calderon_Ouridi_Kaygorodov_2022}. Hence, we obtain the required description in Lemma~\ref{lemma_AutE1} (see below).

\begin{theo}[Kaygorodov, Volkov~\cite{Kaygorodov_Volkov_2019}] Every two-dimensonal algebra  is isomorphic to one and only one algebra from Table 1.
\end{theo}

\bigskip
\begin{center}
Table 1
\end{center}
$${\tiny\begin{array}{lll}

\hline 

&&\\

\bfA_1(\al)  & M=\matr{e_1+e_2}{\al e_2}{(1-\al)e_2}{0} & g = \matr{1}{0}{a}{1}\\

\bfA_2 & M=\matr{e_2}{e_2}{-e_2}{0} & g=\matr{1}{0}{a}{1}\\

\bfA_3 &  M=\matr{e_2}{0}{0}{0} & g=\matr{b}{0}{a}{b^2} \\

\bigfrac{\bfA_4(\al),}{\al\in\FF_{\geqslant 0}} &  M=\matr{\al e_1 + e_2}{e_1 + \al e_2}{-e_1}{0} &  g=\matr{-1}{0}{0}{1}, \text{ if }\al=0 \\

\bfB_1(\al) & M=\matr{0}{(1-\al)e_1+e_2}{\al e_1-e_2}{0} & \\

\bfB_2(\al) & M=\matr{0}{(1-\al)e_1}{\al e_1}{0} & g=\matr{b}{0}{0}{1}\\

\bfB_3 & M=\matr{0}{e_2}{-e_2}{0} & g=\matr{1}{0}{a}{b}\\

\bigfrac{\bfC(\al,\be),}{\be\in \FF_{\geqslant 0}} & M=\matr{e_2}{(1-\al)e_1 + \be e_2}{\al e_1 - \be e_2}{e_2} & g=\matr{-1}{0}{0}{1}, \text{ if }\be=0\\

\bfD_1(\al,\be), (\al,\be)\in \calU & M=\matr{e_1}{(1-\al)e_1 + \be e_2}{\al e_1 - \be e_2}{0} & g=\matr{1}{1}{0}{-1}, \text{ if }\be=2\al - 1\\

\bfD_2(\al,\be), (\al,\be)\not\in \calT & M=\matr{e_1}{\al e_2}{\be e_2}{0} & g=\matr{1}{0}{0}{b}\\

\bfD_3(\al,\be),  (\al,\be)\not \in \calT  & M=\matr{e_1}{e_1 + \al e_2}{-e_1 + \be e_2}{0} & \\

\bigfrac{\bfE_1(\al,\be,\ga,\de),}{(\al,\be,\ga,\de)\in\calV} & M=\matr{e_1}{\al e_1 + \be e_2}{\ga e_1 + \de e_2}{e_2} &
\bigfrac{g=J, \text{ if }(\al,\ga)=(\de,\be)\neq (-1,-1)}{g\in \Sym_3, \text{if }(\al,\ga)=(\de,\be)= (-1,-1)} \\

\bigfrac{\bfE_2(\al,\be,\ga),}{(\be,\ga)\not\in \calT} & M=\matr{e_1}{(1-\al)e_1 + \be e_2}{\al e_1 + \ga e_2}{e_2} &\\

\bigfrac{\bfE_3(\al,\be,\ga),}{\ga\in\FF^{\times}_{>1}}  & M=\matr{e_1}{(1-\al)\ga e_1 + \frac{\be}{\ga} e_2}{\al\ga e_1 + \frac{1-\be}{\ga}e_2}{e_2} & g=J, \text{ if }\ga=-1 \text{ and } \al=\be\\

\bfE_4 & M=\matr{e_1}{e_1+e_2}{0}{e_2} & \\

\bfE_5(\al) & M=\matr{e_1}{(1-\al)e_1 + \al e_2}{\al e_1 + (1-\al)e_2}{e_2} & g=\matr{a}{c}{1-a}{1-c}\\

\bfN & M=\matr{0}{0}{0}{0} & g\in \GL_2\\

&&\\

\hline
\end{array}}
$$

Here we assume that $\al,\be,\ga,\de\in\FF$ and $a,c\in\FF$, $b\in\FF^{\times}$, $a\neq c$.
\bigskip


In the next lemma we explicitly define the action of the group of automorphisms of ${\bfE_1}(-1,-1,-1,-1)$. 

\begin{lemma}\label{lemma_AutE1}
In case $\Char{\FF}\neq3$, the group of automorphisms for $\mathcal{A}={\bfE_1}(-1,-1,-1,-1)$ is given by
$$\Aut(\mathcal{A})=\left\{\smatr{-1}{0}{-1}{1},\smatr{-1}{1}{-1}{0},\smatr{0}{-1}{1}{-1},\smatr{1}{-1}{0}{-1},\smatr{0}{1}{1}{0},\smatr{1}{0}{0}{1}\right\},$$
which is isomorphic to the symmetric group $\Sym_3$.
\end{lemma}
\begin{proof}
Consider an invertible linear map $g:\algA\to\algA$, given by a matrix
$$g=\matr{a_1}{a_2}{a_3}{a_4}$$
with respect to the basis $\{e_1,e_2\}$ of $\algA$, which was used for the tableau of multiplication of $\algA$ in Table 1, where $a_1,\ldots,a_4\in\FF$. Then $g$ is an automorphism of $\algA$ if and only if for each $1\leqslant i,j\leqslant 2$ we have $g(e_ie_j)-g(e_i)g(e_j) = 0$. These conditions are equivalent to the following system of equalitites:
\begin{eq}\label{eq_sys}
\begin{array}{rcl}
a_1 (1+2 a_3 - a_1) & = & a_3 (1+2 a_1-a_3)  \; = \; 0, \\
a_2(1-a_3)+a_1 (1+a_2-a_4) & = &  a_4 (1-a_1) + a_3 (1 - a_2 + a_4) \; = \;0, \\
a_2 (1+2 a_4-a_2) & = & a_4 (1+2 a_2-a_4) \; = \; 0.  \\
\end{array}
\end{eq}

Assume $a_1=0$.  Then system~(\ref{eq_sys}) implies that  $a_1 = 0$, $a_3 = 1$, and $a_2 = 1 + 2 a_4$. 
\begin{enumerate}
\item[$\bullet$] in case $a_4 = 0$ we obtain the automorphism $g=J$;

\item[$\bullet$] in case $a_4 \neq 0$ system~(\ref{eq_sys}) implies that $a_4 = -1$ 
 and we obtain the automorphism  $g=\matr{0}{-1}{1}{-1}$. 
\end{enumerate}

Assume  $a_1\neq 0$.  Then system~(\ref{eq_sys}) implies that $a_1 = 1 + 2 a_3$.

\begin{enumerate}
\item[$\bullet$] If $a_4=0$, then  system~(\ref{eq_sys}) implies that $a_2 = 1$, $a_3 = -1$ and we obtain the automorphism  $g=\matr{-1}{1}{-1}{0}$.

\item[$\bullet$] Let $a_4\neq 0$. Then system~(\ref{eq_sys}) implies that  $a_4 = 1 + 2 a_2$.

If $a_2=a_3=0$, then we obtain the automorphism  $g={\rm Id}$.

If $a_2=0$ and $a_3\neq0$, then system~(\ref{eq_sys}) implies that $a_3 = -1$ and we obtain the automorphism $g=\matr{-1}{0}{-1}{1}$.

If $a_2\neq 0$ and $a_3=0$, then system~(\ref{eq_sys}) implies that  $a_2 = -1$ and we obtain the automorphism  $g=\matr{1}{-1}{0}{-1}$. 

In case $a_2\neq 0$ and $a_3\neq 0$, system~(\ref{eq_sys}) implies a contradiction.
\end{enumerate}
\end{proof}

\begin{theo}\label{theo_simple}
A two-dimensional algebra is simple if and only if it is isomorphic to one of the following algebras: 
\begin{enumerate}
\item[(1)] $\bf A_4(\alpha)$, $\alpha\in\FF_{\geqslant 0}$,
\item[(2)] $\bf B_1(\alpha)$, 
\item[(3)] $\bf C(\alpha,\beta)$, $\beta\in\FF_{\geqslant 0}$,
\item[(4)] $\bf D_1(\alpha,\beta)$, $(\alpha,\beta)\in\mathcal{U}$ with $\beta\neq 0$,
\item[(5)] $\bf D_3(\alpha,\beta)$, $(\alpha,\beta)\not\in\mathcal{T}$ with $(\alpha,\beta)\neq(0,0)$,
\item[(6)] $\bf E_1(\alpha,\beta,\gamma,\delta)$, $(\alpha,\beta,\gamma,\delta)\in\mathcal{V}$ with $(\alpha,\gamma)\neq(0,0)$, $(\beta,\delta)\neq(0,0)$ and $(\beta,\delta)\neq(1-\alpha,1-\gamma)$,
\item[(7)] $\bf E_2(\alpha,\beta,\gamma)$, $(\beta,\gamma)\not\in\mathcal{T}$ with $(\beta,\gamma)\neq(0,0)$,
\item[(8)] $\bf E_3(\alpha,\beta,\gamma)$, $\gamma\in\FF_{>1}^\times$,
\item[(9)] $\bf E_4$,
\end{enumerate}
where $\al,\be,\ga,\de\in\FF$.
\end{theo}
\begin{proof}
Let $\algA$ be an algebra from Table 1 and $\algA\neq \bfN$. Assume that $\algA$ is not simple, i.e., there exists a one-dimensional ideal $I$ in $\algA$ generated by a non-zero element $x= ae_1+be_2$ for some $a,b\in\FF$.
\begin{enumerate}

\item[1.] Let $\algA=\bf A_4(\alpha)$, where $\alpha\in\FF_{\geqslant 0}$. The set $\{x,e_2x\}$ is linearly dependent if and only if $ab=0$. We also have that the set $\{x,e_1x\}$ is linearly dependent set if and only if $a^2-b^2=0$. Thus, $a=b=0$; a contradiction. Therefore, $\bf A_4(\alpha)$ is simple.

\item[2.] Let $\algA=\bf B_1(\alpha)$, where $\al\in\FF$. The set $\{xe_1,e_1x\}$ is linearly dependent if and only if $b=0$. Also, $\{xe_2,e_2x\}$ is a linearly dependent set if and only if $a=0$. Therefore, we have a contradiction, i.e., $\bf B_1(\alpha)$ is simple.

\item[3.] Let $\algA=\bf C(\alpha,\beta)$, where $\al\in\FF$, $\beta\in\FF_{\geqslant 0}$. The set $\{xe_1,e_2x\}$ is linearly dependent if and only if $\alpha(b^2-a^2)=0$. Also, $\{xe_2,e_1x\}$ is a linearly dependent set if and only if $(1-\alpha)(b^2-a^2)=0$. Thus, $a^2=b^2$. Since $\{x,xe_2\}$ is a linearly dependent set if and only if $a(\alpha b + \beta a)=0$ and $\{x,e_2x\}$ is a linearly dependent set if and only if $a((\alpha b + \beta a)-b)=0$, we obtain $a=b=0$; a contradiction. Therefore, $\bf C(\alpha,\beta)$ is simple.

\item[4.] Let $\algA=\bf D_1(\alpha,\beta)$, where $(\al,\be)\in \calU$. The set $\{xe_2,e_2x\}$ is linearly dependent if and only if $\beta a^2=0$. Also, $\{xe_1,e_1x\}$ is a linearly dependent set if and only if $\beta b(2a+b)=0$. If $\beta\neq 0$, then we obtain $a=b=0$; a contradiction. Therefore, $\bf \bfD_1(\alpha,\beta)$ is simple for $\beta\neq 0$.

\item[5.] Let $\algA=\bf D_3(\alpha,\beta)$, where $(\al,\be)\not \in \calT$. Assume $(\alpha,\beta)\neq(0,0)$. The set $\{xe_2,e_2x\}$ is linearly dependent if and only if $(\alpha+\beta) a^2=0$. 

If $a=0$, then $\{x,xe_1\}$ is a linearly dependent set if and only if $b=0$. 

Assume $a\neq0$. Then $\alpha+\beta=0$, and $\{xe_1,xe_2\}$ is a linearly dependent set if and only if $\alpha a^2=0$. Thus, $\al=\be=0$; a contradiction. Therefore, $\bf D_3(\alpha,\beta)$ is simple for $(\alpha,\beta)\neq(0,0)$.

\item[6.] Let $\algA=\bf E_1(\alpha,\beta,\gamma,\delta)$, where $(\al,\be,\ga,\de)\in\calV$ and $(\alpha,\gamma)\neq(0,0)$, $(\beta,\delta)\neq(0,0)$, $(\beta,\delta)\neq(1-\alpha,1-\gamma)$. 

In case $a=0$, the sets $\{x,x e_1\}$ and $\{x,e_1 x\}$ are linearly dependent if and only if $b=0$. 

In case $b=0$, the sets $\{x,x e_2\}$ and $\{x,e_2 x\}$ are linearly dependent if and only if $a=0$. 

Assume $ab\neq0$. Since $\{x,x e_2\}$ and  $\{x,e_1 x\}$ are linearly dependent sets, then $a+b=0$. The set $\{x,e_1x\}$ is linearly dependent if and only if $\alpha+\beta=1$ and the set $\{x,e_2x\}$ is linearly dependent if and only if $\gamma+\delta=1$; a contradiction. Therefore, in all cases we have a contradiction, i.e., $\bf E_1(\alpha,\beta,\gamma,\delta)$ is simple.

\item[7.] Let $\algA=\bf E_2(\alpha,\beta,\gamma)$, where $\al\in\FF$, $(\be,\ga)\not\in \calT$ and $(\beta,\gamma)\neq(0,0)$. The set $\{x,xe_2\}$ is linearly dependent if and only if $a(\alpha b + \beta a)=0$. 

If $a=0$,  then the sets $\{x,x e_1\}$ and $\{x,e_1 x\}$ are linearly dependent if and only if $b=0$. 

Assume $a\neq0$. Then $\alpha b + \beta a =0$. Hence, the set $\{x,x e_1\}$ is linearly dependent if and only if $ab(1-\be-\ga)=0$, which implies $b=0$ and $\beta=0$. Hence, $\{x,e_2 x\}$ is a linearly dependent set if and only if $\gamma a^2 = 0$; a contradiction. Therefore,  in all cases we have a contradiction, i.e., $\bf E_2(\alpha,\beta,\gamma)$ is simple.

\item[8.] Let $\algA=\bf E_3(\alpha,\beta,\gamma)$, where $\al,\be\in\FF$, $\ga\in\FF^{\times}_{>1}$. Recall that $\gamma\not\in\{0,1\}$. The set $\{x,x^2\}$ is linearly dependent if and only if $(1-\gamma)ab\left( \frac{a}{\gamma}+ b\right)=0$. 

If $a=0$, then the sets $\{x,x e_1\}$ and $\{x,e_1 x\}$ are linearly dependent if and only if $b=0$. 

If $b=0$, the set $\{x,e_2 x\}$ is linearly dependent if and only if $a=0$. 

Assume $ab\neq0$. Then $b=-\frac{a}{\gamma}$. The set $\{x,e_2x\}$ is a linearly dependent set if and only if $\beta=\alpha\gamma$. Hence, the set $\{x,x e_1\}$ is linearly dependent if and only if $\gamma=1$; a contradiction. Therefore,  in all cases we have a contradiction, i.e., $\bf E_3(\alpha,\beta,\gamma)$ is simple.

\item[9.] Let $\algA=\bf E_4$. The set $\{x,e_1x\}$ is linearly dependent if and only if $b=0$. We also have that the set $\{x,xe_2\}$ is a linearly dependent if and only if $a=0$. Therefore, we have a contradiction, i.e., $\bf E_4$ is simple.
\end{enumerate}

For the remaining cases, we consider some one-dimensional ideal $I$ of $\algA$, generated by $x\in\algA$:
\begin{enumerate}
\item[$\bullet$] for $\algA=\bfA_1(\al)$ we take $x= e_2$;

\item[$\bullet$] for $\algA=\bfA_2$, we take $x=e_2$;

\item[$\bullet$] for $\algA=\bfA_3$, we take $x=e_2$;

\item[$\bullet$] for $\algA=\bfB_2(\al)$, we take $x=e_1$;

\item[$\bullet$] for $\algA=\bfB_3$, we take $x=e_2$;

\item[$\bullet$] for $\algA=\bfD_1(\al,0)$ with $(\al,0)\in \calU$, we take  $x=e_1$;

\item[$\bullet$] for $\algA=\bfD_2(\al,\be)$ with $(\al,\be)\not\in \calT$, we take  $x=e_2$;

\item[$\bullet$] for $\algA=\bfD_3(0,0)$, we take  $x=e_1$;

\item[$\bullet$] for $\algA=\bfE_1(\al,\be,\ga,\de)$ with $(\al,\be,\ga,\de)\in\calV$, we take  
\begin{enumerate}
\item[(a)] $x=e_2$ in case $(\alpha,\gamma)=(0,0)$,

\item[(b)] $x=e_1$ in case $(\beta,\delta)=(0,0)$,

\item[(c)] $x= e_1-e_2$ in case $\beta=1-\alpha$, $\delta=1-\gamma$;
\end{enumerate}

\item[$\bullet$] for $\algA=\bfE_2(\al,0,0)$, we take $x=e_1$;

\item[$\bullet$] for $\algA=\bfE_5(\al)$, we take $x=e_1-e_2$.
\end{enumerate}
\end{proof}

Theorem~\ref{theo_simple} implies the following results.

\begin{cor}\label{cor_simple} The automorphism group of a two-dimensional simple algebra is finite.
\end{cor}

\begin{cor}\label{cor_simple2}
A two-dimensional simple algebra $\algA$ has a non-trivial automorphisms group if and only if $\algA$ is isomorphic to one of the following algebras:
\begin{enumerate}
\item[$\bullet$] $\bfA_4(0)$, 

\item[$\bullet$] $\bfC(\al,0)$ for $\al\in\FF$, 

\item[$\bullet$] $\bfD_1(\al,2\al-1)$ for $(\al,2\al-1)\in \calU$ with $\al\neq \frac{1}{2}$,

\item[$\bullet$]  $\bfE_1(\al,\be,\be,\al)$ for $(\al,\be,\be,\al)\in\calV$, $(\al,\be)\neq (-1,-1)$, $(\al,\be)\neq (0,0)$, $\al+\be\neq1$, 

\item[$\bullet$] $\bfE_1(-1,-1,-1,-1)$ when $\Char{\FF}\neq3$,

\item[$\bullet$]  $\bfE_3(\al,\al,-1)$ for $\al\in\FF$.
\end{enumerate}
\end{cor}

\section{Invariants}\label{section_invariants} 

In this section we describe the algebra of polynomial invariants $I_m(\algA)$ for any two-dimensional algebra $\algA$ with non-zero multiplication. We assume that the characteristic of $\FF$ is zero. For short, we denote $x_r:=x_{r1}$ and $y_r:=y_{r2}$ for all $1\leqslant r\leqslant m$. Given $\un{r}=(r_1,\ldots,r_k) \in\{1,\ldots,m\}^k$, we write $x_{\un{r}}=x_{r_1}\cdots x_{r_k}$ and $y_{\un{r}}=y_{r_1}\cdots y_{r_k}$.  Denote 
$\Omega_m$ the set of all pairs $(\un{r},\un{s})$ such that  
$\un{r}=(r_1,\ldots,r_k)$, $\un{s}=(s_1,\ldots,s_l)$ with $k,l\geqslant 0$,  $k+l=m$, $r_1< \cdots <r_k$, $s_1<\cdots<s_l$ and $\{r_1, \ldots,r_k,s_1,\ldots,s_l\} = \{1,\ldots,m\}$. In case $k=0$ ($l=0$, respectively), we denote $\un{r}=\emptyset$ ($\un{s}=\emptyset$, respectively). 

\subsection{Partial cases}\label{section_partial}
Note that by formula~(\ref{eq_action}) an automorphism $g=(g_{ij})_{1\leqslant i,j\leqslant 2}$ of $\Aut(\algA)\leqslant \GL_2$ acts on $\FF[\algA^m]$ as follows:
\begin{eq}\label{eq_action2}
g^{-1} x_{r}=g_{11} x_r + g_{12} y_r \text{ and }g^{-1} y_r = g_{21} x_r + g_{22} y_r 
\end{eq}%
for all $1\leqslant r\leqslant m$.

\begin{prop}\label{prop_aut1}
Assume that $K \subset \FF$ is an infinite subset.
\begin{enumerate}[(a)]
\item[1.] If there exist $a,b, c\in\FF$ such that $\matr{a}{0}{b}{c}$ lies in $\Aut(\algA)$ and $a^t c^{m-t}\neq1$ for all $0\leqslant t\leqslant m-1$, then $I_m(\algA)\subset\FF[x_1,\ldots,x_m]$.   

\item[2.] If there exists $a\in\FF$ such that $\matr{a}{0}{0}{c}$ lies in $\Aut(\algA)$ for all $c\in K$, then $I_m(\algA)\subset\FF[x_1,\ldots,x_m]$.  

\item[3.] If there exists $c\in\FF$ such that $\matr{a}{0}{0}{c}$ lies in $\Aut(\algA)$ for all $a\in K$, then $I_m(\algA)\subset\FF[y_1,\ldots,y_m]$.  
\end{enumerate}
\end{prop}
\begin{proof}

\medskip
\noindent{\bf 1.} Assume that $f\in I_m(\algA)$, $a,b,c\in\FF$ with $a^t c^{m-t}\neq1$ for all $0\leqslant t\leqslant m-1$ and $g^{-1}f=f$, where
$$g=\matr{a}{0}{b}{c}.$$ 
To complete the proof it is enough to show that $f\in\FF[x_1,\ldots,x_m]$. Moreover, it is easy to see that by Proposition~\ref{prop_char0} without loss of generality we can assume that $\mdeg(f)=1^m$.  Then 
$$f=\sum_{(\un{r},\un{s})\in \Omega_m} \al_{\un{r},\un{s}} x_{\un{r}} y_{\un{s}}$$
for some $\al_{\un{r},\un{s}}\in \FF$. By formulas~(\ref{eq_action2}) we have 
$$g^{-1}f=\sum_{(\un{r},\un{s})\in \Omega_m} \al_{\un{r},\un{s}} a^{\#\un{r}}\, x_{\un{r}} (b\, x_{s_1} + c\, y_{s_1}) \cdots (b\, x_{s_l} + c\, y_{s_l}), $$
where $l$ stands for $\#\un{s}$. To complete the proof of part 1, we claim that

\begin{eq}\label{eq_claim2}
\al_{\un{r},\un{s}}=0 \text{ for all }(\un{r},\un{s})\in\Omega_m \text{ with }\#\un{r}<m. 
\end{eq}

We prove claim~(\ref{eq_claim2}) by the increasing induction on $0\leqslant \#\un{r}<m$. Note that since $f=g^{-1}f$, for every monomial $w=x_{\un{r}} y_{\un{s}}$, where $(\un{r},\un{s})\in \Omega_m$, we have that the coefficient of $w$ in $f$ is equal to the coefficient of $w$ in $g^{-1}f$.

Assume $\#\un{r}=0$. Consider $w=y_1\cdots y_m$. Then 
\begin{eq}\label{eq_coef1}
\al_{\emptyset,(1\cdots m)} =  \al_{\emptyset,(1\cdots m)} c^m.
\end{eq}%
Since $c^m\neq 1$, we obtain that $\al_{\emptyset,(1\cdots m)}=0$, i.e., claim~(\ref{eq_claim2}) holds for $\#\un{r}=0$.

Assume $\#\un{r}=1$ for $m>1$, i.e., $\un{r}=(r)$. Consider $w=x_r y_1\cdots \widehat{y_r} \cdots y_m$, where $1\leqslant r\leqslant m$. Since $\al_{\emptyset,(1\cdots m)}=0$, then
\begin{eq}\label{eq_coef2}
\al_{r,(1\cdots \hat{r} \cdots m)} =  a c^{m-1} \al_{r,(1\cdots \hat{r} \cdots m)}.
\end{eq}%
Since $ac^{m-1}\neq 1$, then $\al_{r,(1\cdots \hat{r} \cdots m)}=0$, i.e., claim~(\ref{eq_claim2}) holds for $\#\un{r}=1$.

Given $0<t<m$, assume claim~(\ref{eq_claim2}) holds for all $(\un{r'},\un{s'})\in\Omega_m$ with $\#\un{r}'<t$. Consider $w=x_{\un{r}} y_{\un{s}}$ for some $(\un{r},\un{s})\in\Omega_m$ with $\#\un{r}=t$. Since  $\al_{\un{r'},\un{s'}}=0$ in case $\#\un{r'}<t$, we have
\begin{eq}\label{eq_coef3}
\al_{\un{r},\un{s}} =  a^t c^{m-t} \al_{\un{r},\un{s}}.
\end{eq}%
Since $a^t c^{m-t}\neq 1$, then $\al_{\un{r},\un{s}}=0$,  i.e., claim~(\ref{eq_claim2}) holds for $\#\un{r}=t$. Therefore, claim~(\ref{eq_claim2})  is proven.

\medskip
\noindent{\bf 2.} Assume that $f\in I_m(\algA)$ and for every $c\in K$ we have that $g^{-1}f=f$, where
$$g=\matr{a}{0}{0}{c}.$$ 
To complete the proof it is enough to show that $f\in\FF[x_1,\ldots,x_m]$. As in part 1, without loss of generality we can assume that $\mdeg(f)=1^m$ and 
$$f=\sum_{(\un{r},\un{s})\in \Omega_m} \al_{\un{r},\un{s}} x_{\un{r}} y_{\un{s}}$$
for some $\al_{\un{r},\un{s}}\in \FF$. By formulas~(\ref{eq_action2}) we have
$$g^{-1}f=\sum_{(\un{r},\un{s})\in \Omega_m} \al_{\un{r},\un{s}} a^{\#\un{r}}\, c^{\#\un{s}}\, x_{\un{r}} y_{\un{s}}.$$
Since $f=g^{-1}f$ for all $c\in K$, for every monomial $w=x_{\un{r}} y_{\un{s}}$, where $(\un{r},\un{s})\in \Omega_m$, the coefficient of $w$ in $f$ is equal to the coefficient of $w$ in $g^{-1}f$, i.e.,
$$\al_{\un{r},\un{s}} = a^{\#\un{r}}\, c^{\#\un{s}}\, \al_{\un{r},\un{s}}$$
for every $c\in K$. Therefore, $\al_{\un{r},\un{s}}=0$ in case $\#\un{s}>0$. Hence,  $f\in\FF[x_1,\ldots,x_m]$.

\medskip
\noindent{\bf 3.} The proof is similar to the proof of part 2. 
\end{proof}

\begin{lemma}\label{lemma_A1}
If $\algA$ is $\bfA_1(\al)$ or $\bfA_2$, then the algebra $I_m(\algA)$ is generated by 
$$1,x_1,\ldots,x_m \;\text{  and }\; x_r y_s - y_r x_s \; (1\leqslant r<s\leqslant m);$$ 
\end{lemma}
\begin{proof} Denote the set from the formulation of the lemma by $S_m$. It is easy to see that $S_m\subset I_m(\algA)$. For short, denote $h_{rs}=x_r y_s - y_r x_s$. The statement of lemma is a consequence of Theorem~\ref{theo_Weyl} and the following two claims.

\medskip
\noindent{\it Claim 1}. We have $\Pol_2^m(S_2)\subset \FF{\rm-span}(S_m)$ for every $m>2$.

To prove Claim 1, consider $\Pol_2^m(x_r)=\{x_1,\ldots,x_m\}$,  since we have
$$\Phi_{2,m}(x_r)=a_{r1} x_1 + \cdots + a_{rm} x_m$$
for $r=1,2$ (see Definition~\ref{def_polar} for the details). Similarly, $\Pol_2^m(h_{12})=\{h_{rs}\,|\,1\leqslant r,s\leqslant m\}$, since
$$\begin{array}{cl}
\Phi_{2,m}(h_{12}) & =(a_{11} x_1 + \cdots + a_{1m} x_m) (a_{21}y_1 + \cdots +a_{2m}y_m) \\
                   &-\; (a_{21}x_1 +\cdots + a_{2m}x_m)(a_{11}y_1 + \cdots+ a_{1m} y_m) \\
                    &= \sum\limits_{1\leqslant r,s \leqslant m} a_{1r} a_{2s} h_{rs}.
\end{array}
$$
Claim 1 is proven.

\medskip
\noindent{\it Claim 2}. The set $S_2$ generates $I_2(\algA)$.

We have $g=\matr{1}{0}{a}{1}$. Consider an $\NN^2$-homogeneous invariant $f\in I_2(\algA)$ of multidegree $\De=(\de,\la)$, where we assume that $f$ does not contain a monomial $x_1^{\de} x_2^{\la} \in I_2(\algA)$. We prove by induction on $|\De|>0$ that $f\in\alg\{S_2\}$.

Assume $\De=(\de,0)$ for $\de>0$. Then $f=\sum_{i=1}^{\de} \al_i x_1^{\de-i} y_1^i$ for some $\al_1,\ldots,\al_{\de}\in\FF$ and 
$$g^{-1}f=\sum_{i=1}^{\de} \al_i x_1^{\de-i} (a x_1+y_1)^i=f.$$%
Since the coefficients of $x_1^{\de}$ are $\al_1 a+\al_2 a^2+\cdots + \al_{\de} a^{\de}=0$, we obtain that $f=0$. Similarly, we obtain that if $\De=(0,\la)$ for $\la>0$, then $f=0$.

Assume $\De=(\de,\la)$ for $\de,\la>0$. Then $f=\sum \al_{ij} x_1^{\de-i} y_1^i x_2^{\la-j} y_2^j$
for some $\al_{ij}\in\FF$, where the sum ranges over all $0\leqslant i\leqslant \de$, $0\leqslant j\leqslant \la$ with $(i,j)\neq(0,0)$. Applying equality $x_2y_1 = x_1y_2 -h_{12}$ to monomials of $f$, we can rewrite $f$ as follows
$$f=\sum_{i=1}^{\de} \be_i x_1^{\de-i} y_1^i y_2^{\la} + \sum_{j=1}^{\la} \ga_j x_1^{\de} x_2^{\la-j} y_2^j + h_{12} \tilde{f}$$
for some $\be_j,\ga_j\in\FF$ and $\tilde{f}\in\FF[\algA^2]$ of multidegree $(\de-1,\la-1)$. Note that in the first sum we do not have the case of $i=0$, since otherwise the first and the second sums would contain one and the same monomial $x_1^{\de} y_2^{\la}$. Consider
$$g^{-1}f=\sum_{i=1}^{\de} \be_i x_1^{\de-i} (ax_1+y_1)^i (ax_2+y_2)^{\la} + \sum_{j=1}^{\la} \ga_j x_1^{\de} x_2^{\la-j} (ax_2+y_2)^j + h_{12} (g^{-1}\tilde{f}) = f.$$
The coefficients of $x_1^{\de} x_2^{\la}$ are $\sum_{i=1}^{\de} \be_i a^{i+\la} + \sum_{j=1}^{\la} \ga_j a^j=0$. Therefore, $\be_1=\cdots=\be_{\de}=0$ and $\ga_1=\cdots=\ga_{\la}=0$. Hence, $f=h_{12}\tilde{f}$ and $\tilde{f}\in I_2(\algA)$. The induction hypothesis concludes the proof of claim 2.

%
%
%
\end{proof}

\begin{remark}\label{remark_S2}
Given the group $\Sym_2=\{ {\rm Id}, J\}\leqslant \GL_2$, the algebra $\FF[\algA^m]^{\Sym_2}$ is minimally generated by $1$, $x_r+y_r$ ($1\leqslant r\leqslant m$), $x_r y_s + y_r x_s$ ($1\leqslant r\leqslant s\leqslant m$).
\end{remark}
\begin{proof} By Theorem 2.5 of~\cite{domokos2009vector}, the algebra $\FF[\algA^m]^{\Sym_2}$ is minimally generated by $1$, $x_r+y_r$ ($1\leqslant r\leqslant m$), $x_r x_s + y_r y_s$ ($1\leqslant r\leqslant s\leqslant m$). Since $(x_r+y_r)(x_s+y_s) = (x_r x_s + y_r y_s) + (x_r y_s + y_r x_s)$, the claim of this remark is proven. 
\end{proof}

\begin{lemma}\label{lemma_E1}
If $\algA=\bfE_1(-1,-1,-1,-1)$, then the algebra $I_m(\algA)$ is generated by 
$$1, \quad H_{rs}=2 x_r x_s - x_r y_s - y_r x_s + 2 y_r y_s\quad (1\leqslant r\leqslant s\leqslant m),  $$
$$T_{rst}=2 x_r x_s x_t - x_r x_s y_t - x_r y_s x_t - x_r y_s y_t- y_r x_s x_t - y_r x_s y_t - y_r y_s x_t + 2 y_r y_s y_t \quad
(1\leqslant r\leqslant s\leqslant t\leqslant m). $$
\end{lemma}
\begin{proof} Denote the set from the formulation of the lemma by $S_m$. 
Since $\Aut(\algA)$ is isomorphic to $\Sym_3$ (see Lemma~\ref{lemma_AutE1}) and $\Sym_2=\{{\rm Id}, J\}<\Sym_3$, then we can consider the transfer map  
$$\Upsilon: \FF[\algA^m]^{\Sym_2} \to \FF[\algA^m]^{\Sym_3},$$%
which was defined in Section~\ref{section_finite_group}. By Remark~\ref{remark_S2}, the algebra $\FF[\algA^m]^{\Sym_2}$ is generated by $1$, $f_r:=x_r+y_r$ ($1\leqslant r\leqslant m$), $f_{rs}:=x_r y_s + y_r x_s$ ($1\leqslant r\leqslant s\leqslant m$).
Since $\Upsilon(f_r f_s)=6 H_{rs}$ and $\Upsilon(f_r f_s f_t)=-6 T_{rst}$, we obtain $S_m\subset I_m(\algA)$.

By Table 1 of~\cite{Cziszter_Domokos_Szollosi_2018}, the algebra $I_m(\algA)$ is generated by its $\NN^m$-homogeneous elements of degree $\leqslant4$. Therefore, by Proposition~\ref{prop_char0}, to complete the proof, it suffices to show that 
\begin{eq}\label{eq_claim_Ups}
\text{each } h\in I_m(\algA) \text{ of multidegree } \De=1^m  \text{ belongs to the subalgera of }  I_m(\algA) \text{ generated by }S_m,%
\end{eq}%
for every $1\leqslant m\leqslant 4$. By Remark~\ref{remark_transfer}, in claim~(\ref{eq_claim_Ups}) we can assume that $h=\Upsilon(f)$, where $f$ ranges over some basis $B$ of the $\NN^m$-homogeneous component of $\FF[\algA^m]^{\Sym_2}$ of multidegree $\De=1^m$.

\begin{enumerate}
\item[1.] In case $\De=(1)$, we consider $B=\{f_1\}$ to see that $h=0$, since  $\Upsilon(f_1)=0$.

\item[2.] Let $\De=(11)$. Then consider $B=\{f_1 f_2, f_{12}\}$ and use the equalities $\Upsilon(f_1 f_2)= 6 H_{12}$ and $\Upsilon(f_{12})=2 H_{12}$ to prove claim~(\ref{eq_claim_Ups}) for $m=2$. 

\item[3.] Let $\De=(111)$. Then consider $B\subset \{f_1 f_2 f_3,\; f_1 f_{23},\; f_2 f_{13},\; f_3 f_{12}\}$ and use the equalities $\Upsilon(f_1 f_2 f_3)= - 6 T_{123}$ and $\Upsilon(f_1 f_{23})=-4 T_{123}$, $\Upsilon(f_2 f_{13})=-4 T_{123}$, $\Upsilon(f_3 f_{12})=-4 T_{123}$ to prove claim~(\ref{eq_claim_Ups}) for $m=3$.

\item[4.] Let $\De=(1111)$. Then consider 
$$B\subset \{f_1 f_2 f_3 f_4,\;  f_{rs} f_{tq},\; f_r f_s f_{tq} \,|\, 
\{r,s,t,q\}=\{1,2,3,4\}\}$$
and use the equalities 
$$\begin{array}{rcl}
\frac{1}{3}\Upsilon(f_1 f_2 f_3 f_4) & = & H_{12}H_{34} + H_{13}H_{24} + H_{14}H_{23},\\ 
\frac{1}{2}\Upsilon(f_1 f_2 f_{34})  & = & H_{13}H_{24} + H_{14}H_{23},\\
\frac{3}{2}\Upsilon(f_{12} f_{34})   & = & -H_{12}H_{34} + 2 H_{13}H_{24} + 2 H_{14}H_{23},\\
\end{array}$$
together with equality $H_{rs}=H_{sr}$ and the symmetry of $T_{rsl}$ with respect to permutations of $\{r,s,l\}$ to prove claim~(\ref{eq_claim_Ups}) for $m=4$.
\end{enumerate}
\end{proof}

We will prove the minimality of a generating set for $I_m(\algA)$ using the following remark.

\begin{remark}\label{remark_min}
Assume that the algebra $I_m(\algA)$ is generated by a set $S=\{1,f_1,\ldots,f_d\}$ of $\NN^m$-homogeneous elements, which is {\it multidegree-irreducible}, i.e., for every $1\leqslant i\leqslant d$ we have
\begin{enumerate}
\item[$\bullet$] $f_i\not \in\FF$;

\item[$\bullet$] $\mdeg(f_i) \neq \mdeg(f_{j_1}) +\cdots + \mdeg(f_{j_k})$ for every $j_1,\ldots,j_k\in\{1,\ldots,d\}\backslash\{ i\}$ with $k\geq 1$.
\end{enumerate}
Then $S$ is a minimal generating set for $I_m(\algA)$.
\end{remark}

\subsection{General case}\label{section_general}


\begin{theo}\label{theo_gens} Assume that the characteristic of $\FF$ is zero and $\algA$  is a two-dimensional algebra. If the group of automorphisms of $\algA$ is trivial, then $I_m(\algA)=\FF[\algA^m]$. Otherwise, modulo isomorphism, $\algA$ belongs to the following list, where $\al,\be\in\FF$ and $m>0$:
%
%
%
%
%

\begin{longtable}{l|c}
$\algA:$ & A minimal generating set for the algebra $I_m(\algA)$: \\
\hline
$\bfA_1(\al)$ &   $1,x_1,\ldots,x_m$ and $x_r y_s - y_r x_s$ ($1\leqslant r<s\leqslant m$) \\  
\hline
$\bfA_2$ &  $1,x_1,\ldots,x_m$ and $x_r y_s - y_r x_s$ ($1\leqslant r<s\leqslant m$) \\  
\hline
$\bfA_3$ &  $1$  \\  
\hline
$\bfA_4(0)$ &  $1, x_r x_s$  ($1\leqslant r\leqslant s\leqslant m$), $y_1,\ldots,y_m$   \\  
\hline

$\bfB_2(\al)$ &  $1,y_1,\ldots,y_m$\\  
\hline
$\bfB_3$ &  $1,x_1,\ldots,x_m$ \\ 
\hline

$\bfC(\al,0)$ & $1, x_r x_s$  ($1\leqslant r\leqslant s\leqslant m$), $y_1,\ldots,y_m$   \\  
\hline

$\bfD_1(\al,2\al-1)$,\;\;\;\;\; $(\al,2\al-1)\in \calU$ &  $1$, $2x_r + y_r$ ($1\leqslant r\leqslant m$), $y_r y_s$ ($1\leqslant r\leqslant s\leqslant n$)    \\  
\hline
$\bfD_2(\al,\be)$,\;\;\;\;\; $(\al,\be)\not\in\calT$  & $1,x_1,\ldots,x_m$ \\ 
\hline

$\bfE_1(\al,\be,\be,\al)$, &   $1$,  $x_r+y_r$, $x_r y_r$ ($1\leqslant r\leqslant m$), $x_r y_s + y_r x_s$ ($1\leqslant r<s\leqslant m$)  \\
$(\al,\be,\be,\al)\in\calV$, $(\al,\be)\neq (-1,-1)$ &\\
\hline
$\bfE_1(-1,-1,-1,-1)$&  $1$, $2 x_r x_s - x_r y_s - y_r x_s + 2 y_r y_s$ ($1\leqslant r\leqslant s\leqslant m$),   \\
& $2 x_r x_s x_t - x_r x_s y_t - x_r y_s x_t - x_r y_s y_t-$ \\ 
& $- y_r x_s x_t - y_r x_s y_t - y_r y_s x_t + 2 y_r y_s y_t$\\
& ($1\leqslant r\leqslant s\leqslant t\leqslant m$) \\
\hline
$\bfE_3(\al,\al,-1)$ \, &  $1$,  $x_r+y_r$, $x_r y_r$ ($1\leqslant r\leqslant m$), $x_r y_s + y_r x_s$ ($1\leqslant r<s\leqslant m$)   \\ 
\hline
$\bfE_5(\al)$ &  $1$ and $x_r + y_r$ $(1\leqslant r\leqslant m)$ \\  
\hline
$\bfN$ &  $1$\\  
\hline
\end{longtable}
\end{theo}
\begin{proof}  If the group of automorphisms of $\algA$ is trivial, then Remark~\ref{remark_trivial} completes the proof. Therefore, we assume that $\Aut(\algA)$ is non-trivial. 

We apply Table 1 (see Section~\ref{section_simple2dim}) to see that $\algA$ is isomorphic to an algebra from one of the items of the theorem.  Therefore, we can assume that $\algA$ is an algebra from one of the items of the theorem. Denote by $S$ the subset of $\FF[\algA^m]$, which is claimed to be a generating set for $I_m(\algA)$.  Considering an arbitrary automorphism $g\in \Aut(\algA)$, given in Table 1, and using formulas~(\ref{eq_action2}), it is easy to see that $S\in I_m(\algA)$. Note that $\pi_{\un{r}}(S)\subset \alg\{S\}$ for every $\un{r}\in \calcP^{t}_m$, where $t\geqslant m$. Therefore,  to complete the proof, it is enough to show that each multilinear non-constant invariant
$$f=\sum_{(\un{i},\un{j})\in \Omega_m} \al_{\un{i},\un{j}} x_{\un{i}} y_{\un{j}}  \;\text{ from }\; I_m(\algA)$$
lies in the subalgebra $\alg\{S\}$ generated by $S$ (see Proposition~\ref{prop_char0}), where $\al_{\un{i},\un{j}}\in\FF$. Let $g=g(a,b,c)$ be an arbitrary non-identity element of $\Aut(\algA)$ given in Table 1, where  $a,b,c\in\FF$ are arbitrary elements with $b\neq0$ and $a\neq c$.

\begin{enumerate}
\item[1.] Let $\algA$ be $\bfA_1(\al)$ or $\bfA_2$. Then see Lemma~\ref{lemma_A1}.

\item[2.]  Let $\algA=\bfA_3$. Then $g=\matr{b}{0}{a}{b^2}$. Since $\FF$ is infinite, we can assume that $b^t\neq1$ for all $1\leqslant t\leqslant 2m$. Then we can apply part 1 of Proposition~\ref{prop_aut1} and obtain that $f\in \FF[x_1,\ldots,x_m]$, i.e., $f=\al x_1\cdots x_m$ for some $\al\in\FF$. Since $g^{-1} f = \al b^m x_1\cdots x_m = f$, we have that $\al=0$ and the required statement is proven.

\item[3.]  Let $\algA$ be $\bfA_4(0)$ or $\bfC(\al,0)$. Then $g=\matr{-1}{0}{0}{1}$.  By Lemma~\ref{lemma_diag}, we can assume that $f=x_{\un{i}} y_{\un{j}}$ is a monomial for some $(\un{r},\un{s})\in \Omega_m$. Since  $g^{-1} f = (-1)^{|\un{r}|} f$, we have that $|\un{r}|$ is even and the required statement follows.

\item[4.]  Let $\algA=\bfB_2(\al)$. Then part 3 of Proposition~\ref{prop_aut1} implies that $f\in\FF[y_1,\ldots, y_m]$.

\item[5.]  Let $\algA=\bfB_3$. Then part 1 of Proposition~\ref{prop_aut1} implies that $f\in\FF[x_1,\ldots, x_m]$.

\item[6.] Let $\bfD_1(\al,2\al-1)$ for $(\al,2\al-1)\in \calU$. We have $g=\matr{1}{1}{0}{-1}$ and $|\Aut(\algA)|=2$. Thus the algebra $I_m(\algA)$ is generated by $\NN^m$-homogeneous invariants of degree $\leqslant2$ by Theorem~\ref{theo_Noether}. Hence we can assume that $\deg(f)\leqslant 2$. 

Assume $\deg(f)=1$. Then $f=\al x_1 + \be y_1$ for some $\al,\be\in\FF$ and $g^{-1}f = \al (x_1+y_1) - \be y_1=f$. Thus $\al=2\be$ and $f=\be (2x_1 + y_1)$.   

Assume $\deg(f)=2$. Since $y_1y_2\in I_m(\algA)$, we can assume that $f= \al x_1 x_2 +\be_1 x_1 y_2 + \be_2 y_1 x_2 $ for some $\al,\be_1,\be_2\in\FF$. Then
$$g^{-1}f=  \al (x_1+y_1) (x_2+y_2) -\be_1 (x_1+y_1) y_2 - \be_2 y_1 (x_2+y_2)  = f.$$%
Considering the coefficients of $x_1 y_2$, $y_1x_2$, $y_1y_2$, respectively, we obtain that $\al-\be_1 = \be_1$, $\al-\be_2 = \be_2$, $\al-\be_1 - \be_2=0$. Thus 
$$f=\be_1 (2x_1x_2 + x_1 y_2 + y_1 x_2) = \frac{\be_1}{2}\bigg((2x_1 + y_1)(2x_2 + y_2) - y_1y_2\bigg).$$
The claim is proven.

\item[7.] Let $\algA=\bfD_2(\al,\be)$ for $(\al,\be)\not\in\calT$. Then part 2 of Proposition~\ref{prop_aut1} implies that $f\in\FF[x_1,\ldots, x_m]$.

\item[8.] Let $\algA=\bfE_1(\al,\be,\be,\al)$ for $(\al,\be,\be,\al)\in\calV$, $(\al,\be)\neq (-1,-1)$ or $\algA=\bfE_3(\al,\al,-1)$.  We have $g=\matr{0}{1}{1}{0}$ and $\Aut(\algA)$ acts on $\FF[\algA^m]$ as the symmetric group $\Sym_2$ by permutation of $x_i$ and $y_i$, i.e., $I_m(\algA)=\FF[\algA^m]^{\Sym_2}$. Since the generators for $\FF[\algA^m]^{\Sym_2}$ are well known (as an example, see Remark~\ref{remark_S2}), the required statement is proven.

\item[9.] Let $\algA=\bfE_1(-1,-1,-1,-1)$. Then see Lemma~\ref{lemma_E1}.

\item[10.] Let $\algA=\bfE_5(\al)$. Then $g=\matr{a}{c}{1-a}{1-c}$ for any $a,c\in\FF$ with $a\neq c$. For $z_i:=x_i+y_i$, where $1\leqslant i\leqslant m$, we have  $\FF[\algA^m]=\FF[z_1,x_1,\ldots,z_m,x_m]$. Since $g^{-1} z_r = z_r$ and $g^{-1} x_r = (a-c) x_r + c z_r$, then $I_m(\algA)=\FF[z_1,x_1,\ldots,z_m,x_m]^H$, where the group $H$ consists of matrices $\matr{1}{0}{a}{b}$ for all $a\in\FF$, $b\in\FF^{\times}$.   By part 1 of Proposition~\ref{prop_aut1} we have that $f\in\FF[z_1,\ldots, z_m]$.

\item[11.] If $\algA=\bfN$, then obviously $I_m(\bfN)=\FF$.
\end{enumerate}

For each of the above-considered cases the minimality of $S$ follows from the fact that $S$ is multidegree-irreducible (see Remark~\ref{remark_min}).
\end{proof}

\subsection{Trace invariants}\label{section_trace_inv}

Writing down $M_{ij}$ as $M_{ij}=(M_{ij1},M_{ij2})$, we can see that formulas~(\ref{eq_trL_trR}) imply that for all $1\leqslant r\leqslant m$ we have:
\begin{eq}\label{eq_trL3}
\tr(\chi_r\chi_0)= (\al_1 + \al_2) x_r + (\be_1+\be_2) y_r 
\;\text{ for }\;
M=\matr{(\al_1,\ast)}{(\ast,\al_2)}{(\be_1,\ast)}{(\ast,\be_2)},
\end{eq}
\begin{eq}\label{eq_trR3}
\tr(\chi_0\chi_r)= (\al_1 + \al_2) x_r + (\be_1+\be_2) y_r 
\;\text{ for }\;
M=\matr{(\al_1,\ast)}{(\be_1,\ast)}{(\ast,\al_2)}{(\ast,\be_2)}.
\end{eq}

\medskip
\begin{prop}\label{prop_tr_dim2}
For every $\al,\be,\ga,\de\in\FF$, $m>0$ and $1\leqslant r,s\leqslant m$ we have the following trace formulas:
\begin{longtable}{l|lll}
\hline
$\bfA_1(\al)$ &  $\tr(\chi_r \chi_0) =(1+\al) x_r $ &  $\tr(\chi_0 \chi_r)= (2-\al) x_r$    \\  
&   $\tr(\chi_r(\chi_s \chi_0)) = (1+\al^2) x_r x_s $ & $\tr((\chi_s \chi_0) \chi_r) =(1+\al-\al^2) x_r x_s$\\
&   $\tr(\chi_r(\chi_0 \chi_s)) = (1+\al-\al^2) x_r x_s$  & $\tr((\chi_0 \chi_s) \chi_r) =(2-2 \al+\al^2) x_r x_s$ \\
&   $\tr((\chi_r \chi_s)\chi_0) = (1+\al) x_r x_s$  & 
$\tr(\chi_0 (\chi_r \chi_s)) =(2 -\al) x_r x_s$ \\
\hline
$\bfA_2$ &  $\tr(\chi_r \chi_0) = x_r $ &  $\tr(\chi_0 \chi_r)= -x_r$    \\  
&   $\tr(\chi_r(\chi_s \chi_0)) = x_r x_s$ & $\tr((\chi_s \chi_0) \chi_r) = -x_r x_s$\\
&   $\tr(\chi_r(\chi_0 \chi_s)) =  -x_r x_s$& $\tr((\chi_0 \chi_s) \chi_r) = x_r x_s$ \\
&   $\tr((\chi_r \chi_s)\chi_0) = 0$  & 
$\tr(\chi_0 (\chi_r \chi_s)) = 0$ \\
\hline
$\bfA_3$ &  $\tr(\chi_r \chi_0) =0 $ &  $\tr(\chi_0 \chi_r)=0 $    \\  
\hline
$\bfA_4(0)$  &  $\tr(\chi_r \chi_0) = -y_r$ &  $\tr(\chi_0 \chi_r)= y_r$    \\ 
&   $\tr(\chi_r(\chi_s \chi_0)) = 2 x_r x_s + y_r y_s$ & \\
\hline

$\bfB_2(\al)$ &  $\tr(\chi_r \chi_0) = \al y_r$ &  $\tr(\chi_0 \chi_r)=(1-\al) y_r $    \\  
\hline
$\bfB_3$ &  $\tr(\chi_r \chi_0) = x_r$ &  $\tr(\chi_0 \chi_r)=-x_r$    \\ 
\hline

$\bfC(\al,0)$  &  $\tr(\chi_r \chi_0) = (1+\al) y_r$ &  $\tr(\chi_0 \chi_r)= (2-\al) y_r$    \\  
&  $\tr(\chi_r(\chi_0 \chi_s)) =  x_r x_s + (1+(1-\al)\al)y_r y_s$& \\
\hline

$\bfD_1(\al,2\al-1)$,\;\;\;\;\;  &  
$\tr(\chi_r \chi_0) = \al(2 x_r + y_r)$ &  
$\tr(\chi_0 \chi_r) = (1-\al)(2 x_r + y_r)$    \\
$(\al,2\al-1)\in \calU$ &\multicolumn{2}{l}{$\tr((\chi_0 \chi_s) \chi_r) =
\frac{1}{2}(1-2 \al+2 \al^2) (2 x_r+y_r)(2 x_s+y_s) +(\frac{1}{2}-\al) y_r y_s$}\\
&   \multicolumn{2}{l}{ $\tr((\chi_r \chi_s) \chi_0) =  \frac{\al}{2}(2x_r+y_r)(2x_s+y_s) - \frac{\al}{2}y_r y_s$}& \\
\hline
$\bfD_2(\al,\be)$, &  $\tr(\chi_r \chi_0) = (1+\al) x_r$ &  $\tr(\chi_0 \chi_r)=(1+\be) x_r $    \\ 
$(\al,\be)\not\in\calT$&& \\
\hline
$\bfD_3(\al,\be)$,   & $\tr(\chi_r \chi_0) =   (1+\al) x_r - y_r $ &  $\tr(\chi_0 \chi_r)= (1+\be) x_r + y_r$    \\  
$(\al,\be)\not\in\calT$ && \\
\hline

$\bfE_1(\al,\be,\be,\al)$, &  $\tr(\chi_r \chi_0) = (1+\be)(x_r+y_r)$ &  $\tr(\chi_0 \chi_r)= (1+\al)(x_r+ y_r)$    \\  
$(\al,\be,\ga,\de)\in\calV$,  & 
\multicolumn{2}{l}{$\tr(\chi_r(\chi_s \chi_0))  = (1+\be^2) (x_r+y_r)(x_s+y_s) +
(\al^2-\be^2+2 \be - 1) (x_r y_s + y_r x_s)$}  \\
$(\al,\be)\neq (-1,-1)$ &\multicolumn{2}{l}{$\tr((\chi_0 \chi_s) \chi_r) = (1+\al^2) (x_r+y_r)(x_s+y_s) +
(\be^2-\al^2+2 \al - 1) (x_r y_s + y_r x_s)$}\\
\hline
$\bfE_1(-1,-1,-1,-1)$ &  $\tr(\chi_r \chi_0) = 0$ &  $\tr(\chi_0 \chi_r)= 0$    \\  
 &   $\tr(\chi_r(\chi_s \chi_0)) = H_{rs}$ & \\
& $\tr(\chi_r((\chi_s \chi_t) \chi_0)) = T_{rst}$&\\
\hline

$\bfE_3(\al,\al,-1)$ &  $\tr(\chi_r \chi_0) = (1-\al)(x_r+y_r)$ &  
$\tr(\chi_0 \chi_r)=\al(x_r+y_r)$    \\ 
&\multicolumn{2}{l}{$\tr(\chi_r(\chi_s \chi_0))  = 
(1+\al^2) (x_r+y_r)(x_s+y_s) - 4\al (x_r y_s + y_r x_s)$}  \\
&\multicolumn{2}{l}{$\tr((\chi_0 \chi_s) \chi_r)  = 
(\al^2- 2\al + 2) (x_r+y_r)(x_s+y_s) + 4(\al-1)(x_r y_s + y_r x_s)$}  \\
\hline
$\bfE_5(\al)$ &  $\tr(\chi_r \chi_0) =(1+\al)(x_r+y_r)$ &  $\tr(\chi_0 \chi_r)= (2-\al)(x_r+ y_r)$    \\  
\hline
$\bfN$ &  every operator trace is zero    \\  
\hline
\end{longtable}

Here, the polynomials $H_{rs}$ and $T_{rst}$ were defined in Lemma~\ref{lemma_E1}.
\end{prop}
\begin{proof} We take the tableaux of multiplication for algebras from Table 1. To calculate the traces we apply equalities (\ref{eq_trL3}), (\ref{eq_trR3}) and Propositions~\ref{prop_trace} and~\ref{prop_trace_double}.
\end{proof}

\begin{theo}\label{theo_API}
Assume that the characteristic of $\FF$ is zero, $\algA$ is a two-dimensional algebra with a non-trivial automorphism group and $m>0$. Then the Artin--Procesi--Iltyakov Equality holds for $\algA^m$ if and only if $\algA$ is not isomorphic to any of the following algebras:
\begin{enumerate}
\item[$\bullet$] $\bfA_1(\al)$ with  $m>1$, $\al\in\FF$;

\item[$\bullet$] $\bfA_2$ with $m>1$;


\item[$\bullet$] $\bfD_2(-1,-1)$.
\end{enumerate}
\end{theo}
\begin{proof}  Without loss of generality, we can assume that $\algA$ is an algebra from one of the items of Theorem~\ref{theo_gens}. Denote by $S$ the generating set for $I_m(\algA)$ from Theorem~\ref{theo_gens}.  Hence, the Artin--Procesi--Iltyakov Equality holds for $\algA^m$ if and only if 
\begin{eq}\label{eq_cond}
S\subset \Tr(\algA)_m.
\end{eq}%
\noindent{}For every $d>0$, denote by $S_d$ all elements from $S$ of degree $d$. Similarly,  
for every multidegree $\De$, denote by $S_{\De}$ all elements from $S$ of multidegree $\De$.  

As above, $\beta=\be(I_m(\algA))$ is the maximal degree of elements from $S$. Note that $\be\leqslant 2$, unless $\algA=\bfE_1(-1,-1,-1,-1)$. Obviously, if condition~(\ref{eq_cond}) holds for $m=\be$, then  condition~(\ref{eq_cond}) holds for every $m\geqslant\be$. Note that 

\begin{enumerate}
\item[(a)] $S_1\subset\Tr(\algA)_m$ if and only if $S_{(1)}$ lies in the $\FF$-span of $\tr(\chi_1 \chi_0)$, $\tr(\chi_0 \chi_1)$.

\item[(b)] $S_2\subset\Tr(\algA)_m$ if and only if  
\begin{enumerate}
\item[$\bullet$] $S_{(2)}$ lies in the $\FF$-span of $\tr(\chi_1 \chi_0)^2$, $\tr(\chi_0 \chi_1)^2$, $\tr(\chi_1 \chi_0)\tr(\chi_0 \chi_1)$,  $\tr(\chi_1(\chi_1 \chi_0))$, $\tr(\chi_1(\chi_0 \chi_1))$,  $\tr((\chi_1 \chi_0) \chi_1)$, $\tr((\chi_0 \chi_1) \chi_1)$, $\tr((\chi_1 \chi_1) \chi_0)$, $\tr( \chi_0(\chi_1 \chi_1))$;

\item[$\bullet$] $S_{(11)}$ lies in the $\FF$-span of 
$\tr(\chi_1 \chi_0)\tr(\chi_2 \chi_0)$, 
$\tr(\chi_0 \chi_1)\tr(\chi_0 \chi_2)$,  
$\tr(\chi_r \chi_0)\tr(\chi_0 \chi_s)$,   
$\tr(\chi_r(\chi_s \chi_0))$,
$\tr(\chi_r(\chi_0 \chi_s))$, 
$\tr((\chi_s \chi_0) \chi_r)$, 
$\tr((\chi_0 \chi_s) \chi_r)$,
 $\tr((\chi_s \chi_r) \chi_0)$, 
 $\tr( \chi_0(\chi_r \chi_s))$ for every $r,s$ with $\{r,s\}=\{1,2\}$, in case $m>1$.
\end{enumerate}
\end{enumerate}

In what follows, we will use Proposition~\ref{prop_tr_dim2} and the above observations without reference to them. Since $\Aut(\algA)$ is not trivial, one of the following cases holds.
\begin{enumerate}
\item[1.] Let $\algA$ be $\bfA_1(\al)$ or $\bfA_2$. Then $S_{(1)}=\{x_1\}$ is a subset of $\Tr(\algA)_m$. Thus in case $m=1$ we have that condition~(\ref{eq_cond}) holds. On the other hand, in case $m>1$ we have that $x_1y_2-y_1x_2\not\in \Tr(\algA)_m$.

\item[2.] Let $\algA=\algA_3$. Then $\{1\}=S\subset \Tr(\algA)_m=\FF$.

\item[3.] Let $\algA$ be one of the following algebras: $\bfA_4(0)$, $\bfB_2(\al)$, $\bfB_3$, $\bfE_1(-1,-1,-1,-1)$, $\bfE_5(\al)$, where $\al\in\FF$. Then  condition~(\ref{eq_cond}) holds. 

\item[4.] Let $\algA=\bfC(\al,0)$ for $\al\in\FF$. Then $S_{(1)}=\{y_1\}$ and $S_{(2)}=\{x_1^2\}$ are subsets of $\Tr(\algA)_m$. Moreover, $S_{(11)}=\{x_1 x_2\}$ is a subset of $\Tr(\algA)_m$ in case $m>1$.

\item[5.] Let $\algA=\bfD_1(\al,2\al-1)$ for $(\al,2\al-1)\in \calU$. Then
$S_{(1)}=\{2x_1 + y_1\}$ is a subset of $\Tr(\algA)_m$. Considering 
$\tr((\chi_0 \chi_s) \chi_r)$ and $\tr((\chi_r \chi_s) \chi_0)$ we obtain that $y_r y_s\in \Tr(\algA)_m$ for all $1\leqslant r\leqslant s\leqslant m$, i.e.,  condition~(\ref{eq_cond}) holds.

\item[6.] Let $\algA=\bfD_2(\al,\be)$ for $(\al,\be)\not\in\calT$. In case $(\al,\be)\neq (-1,-1)$, the set $S_{(1)}=\{x_1\}$  is a subset of $\Tr(\algA)_m$. On the other hand, in case $\al=\be=-1$ the set  $S_{(1)}=\{x_1\}$ does not lie in $\Tr(\algA)_m$.

\item[7.] Let $\algA=\bfE_1(\al,\be,\be,\al)$ for $(\al,\be,\ga,\de)\in\calV$, $(\al,\be)\neq (-1,-1)$. Then, 
$S_{(1)} = \{x_1+y_1\}$ is a subset of $\Tr(\algA)_m$. Since $\al+\be\neq 1$, we consider $\tr(\chi_r(\chi_s \chi_0)) + \tr((\chi_0 \chi_s) \chi_r)$ to 
obtain that $x_r y_s + y_r x_s\in \Tr(\algA)_m$ for all $1\leqslant r\leqslant s\leqslant m$, i.e.,  condition~(\ref{eq_cond}) holds.


\item[8.] Let $\algA=\bfE_3(\al,\al,-1)$ for $\al\in\FF$.  Hence, $x_r+y_r$ lies in $\Tr(\algA)_m$ and, therefore, $x_ry_s+y_r x_s$ lies in $\Tr(\algA)_m$ for all $1\leqslant r\leqslant s\leqslant m$, i.e.,  condition~(\ref{eq_cond}) holds.
\end{enumerate}
\end{proof}

\begin{cor}\label{cor_main}
Assume that $\algA$ is a two-dimensional simple algebra with a non-trivial automorphism group and $m>0$. Then the Artin--Procesi--Iltyakov Equality holds for $\algA^m$.
\end{cor}

\begin{remark}\label{rem_new}
Over a field of characteristic zero, there are infinitely many non-isomorphic 2-dimensional algebras $\algA$ with infinite groups of automorphisms (namely, algebras $\bfA_1(\al)$ and $\bfA_2$ for $\al\in\FF$) such that  Artin-Procesi--Iltyakov Equality holds for $\algA$, but  Artin-Procesi--Iltyakov Equality does not hold for $\algA^m$ for all $m>1$.
\end{remark}

\begin{lemma}\label{lemma_non_API}
Over a field of characteristic zero, the algebra $\algA=\bfD_3(\al,-\al-2)$ with $\al\in\FF$ is a simple algebra with the trivial automorphism group such that Artin-Procesi--Iltyakov Equality does not hold for $\algA^m$ for all $m>0$.
\end{lemma}
\begin{proof} By Theorem~\ref{theo_simple} the algebra $\algA$ is simple. Since $\Aut(\algA)$ is trivial, $x_1,y_1$ belong to $I_m(\algA)$ by Remark~\ref{remark_trivial}. On the other hand, the $\FF$-span of $\tr(\chi_1 \chi_0)$, $\tr(\chi_0 \chi_1)$ does not contain $x_1$ and $y_1$ (see Proposition~\ref{prop_tr_dim2}).
\end{proof}

\section{Corollaries}\label{section_cor}

Assume that $\algA$ is an $n$-dimensional algebra and  $\varphi:\algA^2\to\FF$ is a bilinear form over $\algA$. Obviously, we can  consider $\varphi$ as an element of $\FF[\algA^2]$ of multidegree $(1,1)$. The form $\varphi$ is {\it invariant} if $\varphi(g\cdot a,g\cdot b)=\varphi(a,b)$ for all $g\in\Aut(\algA)$ and $a,b\in\algA$, i.e., $\varphi\in I_2(\algA)$.  The form $\varphi$ is called {\it symmetric} if $\varphi(a,b)=\varphi(b,a)$ for all $a,b\in\algA$. Similarly, the form $\varphi$ is called {\it skew-symmetric} if $\varphi(a,b)=-\varphi(b,a)$ for all $a,b\in\algA$. The form $\varphi$ is {\it associative} if $\varphi(ac,b)=\varphi(a,cb)$ for all $a,b,c\in\algA$. As an example, the bilinear form
$$\varphi:M_n^2\to\FF \quad  \text{ defined by } \quad \varphi(A,B)=\tr(AB)$$
for all $A,B\in M_n$ is a symmetric associative invariant nondegenerate bilinear form over $M_n$.

\begin{prop}\label{prop_form}
Assume that the characteristic of $\FF$ is zero and $\algA$ is a two-dimensional algebra with a non-trivial group $\Aut(\algA)$.
\begin{enumerate}
\item[(a)] The algebra $\algA $ has a symmetric invariant nondegenerate bilinear form if and only if $\algA$ is isomorphic to one of the following algebras:
\begin{enumerate}
\item[$\bullet$] $\bfA_4(0)$, 

\item[$\bullet$] $\bfC(\al,0)$ for $\al\in\FF$, 

\item[$\bullet$] $\bfD_1(\al,2\al-1)$ for $(\al,2\al-1)\in \calU$,

\item[$\bullet$]  $\bfE_1(\al,\be,\be,\al)$ for $(\al,\be,\be,\al)\in\calV$, $(\al,\be)\neq (-1,-1)$,

\item[$\bullet$] $\bfE_1(-1,-1,-1,-1)$,

\item[$\bullet$]  $\bfE_3(\al,\al,-1)$ for $\al\in\FF$.
\end{enumerate}

\item[(b)] The algebra $\algA$ has a skew-symmetric invariant nondegenerate bilinear form if and only if $\algA$ is isomorphic to one of the following algebras: 
\begin{enumerate}
\item[$\bullet$] $\bfA_1(\al)$, 

\item[$\bullet$] $\bfA_2$.
\end{enumerate}

\end{enumerate}
\end{prop}
\begin{proof} We can assume that $\algA$ is an algebra from one of the items of Theorem~\ref{theo_gens}. Suppose that $\algA$ has an invariant bilinear form $\varphi$, i.e, $\varphi\in I_2(\algA)$ and $\mdeg(\varphi)=(1,1)$. An $\NN^2$-homogeneous generating set for the algebra $I_2(\algA)$ is described by Theorem~\ref{theo_gens}.

\begin{enumerate}
\item[1.] In case $\algA$ is $\bfA_1(\al)$ or $\bfA_2$ we have that 
$$\varphi = \xi x_1 x_2 + \eta (x_1 y_2 - y_1 x_2)$$ 
for some $\xi,\eta\in\FF$. Since the matrix of the bilinear form $\varphi$ is $\matr{\xi}{\eta}{-\eta}{0}$, the bilinear form $\varphi$ is nondegenerate if and only if $\eta\neq0$. For non-zero $\eta$ the bilinear form $\varphi$ is skew-symmetric in case $\xi=0$, but it is not symmetric. 

\item[2.] In case $\algA=\bfA_3$ we have $\varphi=0$.

\item[3.] In case $\algA=\bfA_4(0)$ or $\algA=\bfC(\al,0)$ for $\al\in\FF$ we have that
$$\varphi = \xi x_1x_2 + \eta y_1y_2$$%
for some $\xi,\eta\in\FF$. Hence, $\varphi$ is symmetric. Moreover, $\varphi$ is nondegenerate in case $\xi,\eta\in\FF^{\times}$.

\item[4.] In case $\algA=\bfB_2(\al)$ for $\al\in\FF$ we have that $\varphi=\xi y_1 y_2$ for some $\xi\in\FF$. Hence,  the bilinear form $\varphi$ is degenerate. The cases of $\algA=\bfB_3$ and $\algA=\bfD_2(\al,\be)$ for $(\al,\be)\not\in\calT$ are similar.

\item[5.] In case $\algA=\bfD_1(\al,2\al-1)$ for $(\al,2\al-1)\in \calU$  we have that
$$\varphi = \xi (2x_1 + y_1)(2x_2+y_2) + \eta y_1y_2$$%
for some $\xi,\eta\in\FF$. Since the matrix of the bilinear form $\varphi$ is $\matr{4\xi}{2
\xi}{2\xi}{\xi+\eta}$, the bilinear form $\varphi$ is nondegenerate if and only if $\xi,\eta\in\FF^{\times}$. For non-zero $\xi$ the bilinear form $\varphi$ is symmetric, but it is not skew-symmetric.

\item[6.] Let $\algA=\bfE_1(\al,\be,\be,\al)$ for $(\al,\be,\be,\al)\in\calV$, $(\al,\be)\neq (-1,-1)$. Then  
$$\varphi = \xi(x_1+y_1)(x_2+y_2) + \eta(x_1y_2 + y_1x_2)$$%
for some $\xi,\eta\in\FF$. Since the matrix of the bilinear form $\varphi$ is 
$\matr{\xi}{\xi+\eta}{\xi+\eta}{\xi}$, the bilinear form $\varphi$  is nondegenerate if and only if $\eta\neq0$ and $2\xi+\eta\neq0$. The bilinear form $\varphi$ is symmetric. Moreover, it is skew-symmetric if and only if $\xi=-\eta =0$. 
The case of $\algA=\bfE_3(\al,\al,-1)$ for $\al\in\FF$ is similar. 

\item[7.] In case $\algA=\bfE_1(-1,-1,-1,-1)$ we have that
$$\varphi=\xi ( 2x_1x_2 - x_1 y_2 - y_1 x_2 + 2y_1 y_2)$$
for $\xi\in\FF$. Since the matrix of the bilinear form $\varphi$ is 
$\matr{2\xi}{-\xi}{-\xi}{2\xi}$, in case $\xi\neq0$ the bilinear form $\varphi$  is nondegenerate and symmetric, but it is not skew-symmetric.

\item[8.] In case $\algA=\bfE_5(\al)$ for $\al\in\FF$  we have that $\varphi = \xi  (x_1 + y_1)(x_2 + y_2)$ for some $\xi\in\FF$. Hence, the bilinear form $\varphi$ is degenerate.
\end{enumerate}
\end{proof}

Since every algebra with the trivial automorphism group has a symmetric invariant nondegenerate bilinear form, Proposition~\ref{prop_form} implies the following corollary.

\begin{cor}\label{cor_form} 
Assume that the characteristic of $\FF$ is zero and $\algA$ is a two-dimensional algebra. 
\begin{enumerate}
\item[1.] If the group $\Aut(\algA)$ is infinite, then the algebra $\algA$ does not admit a symmetric invariant nondegenerate bilinear form.

\item[2.] If the algebra $\algA$ is simple, then $\algA$ admits a symmetric invariant nondegenerate bilinear form.
\end{enumerate}
\end{cor}

\begin{remark}\label{rem_form} 
Assume that the characteristic of $\FF$ is zero. Then a two-dimensional algebra $\algA$ may not admit a symmetric associative invariant nondegenerate bilinear form. As an example, we can take $\algA=\bfA_4(0)$.  
\end{remark}

\bibliographystyle{abbrvurl}
\bibliography{literature}

\end{document}